\documentclass[12pt]{amsart}
\usepackage[utf8]{inputenc}
\usepackage{amsmath,amsthm,amsfonts,amssymb}
\usepackage[all]{xy}
\usepackage[letterpaper, margin=1in]{geometry}
\usepackage{xcolor}
\usepackage{tikz}
    \usetikzlibrary{arrows,automata,positioning}
\usepackage{hyperref}
    \hypersetup{
    colorlinks=true,
    urlcolor=blue,
    citecolor=blue,
    linkcolor=purple
    }

\theoremstyle{plain} 
\newtheorem{thm}{Theorem}[section]
\newtheorem{prop}[thm]{Proposition}
\newtheorem{cor}[thm]{Corollary}
\newtheorem{lemma}[thm]{Lemma}

\theoremstyle{definition} 
\newtheorem{defn}[thm]{Definition}
\newtheorem{exmp}[thm]{Example}
\newtheorem{rmk}[thm]{Remark}
\newtheorem{notation}[thm]{Notation}

\newcommand{\pdim}{\textnormal{pd}}
\newcommand{\reg}{\operatorname{reg}}
\newcommand{\init}{\text{in}_>}
\newcommand{\F}{\mathbb{F}}
\newcommand{\Hilb}{\text{Hilb}}

\makeatletter
\@namedef{subjclassname@2020}{
  \textup{2020} Mathematics Subject Classification}
\makeatother

\thanks{L.\ Ballard was partially supported by the National Science Foundation (DMS-1003384).}

\begin{document}

\title{Properties of the Toric Rings of a Chordal Bipartite Family of Graphs}
    \author[L.\ Ballard]{Laura Ballard}
  \address{Mathematics Department, Syracuse University, Syracuse, NY 13244, U.S.A.}
  \email{lballard@syr.edu} 

  

\begin{abstract}
This work concerns the study of properties of a group of Koszul algebras coming from the toric ideals of a chordal bipartite infinite family of graphs (alternately, these rings may be interpreted as coming from determinants of certain ladder-like structures). We determine a linear system of parameters for each ring and explicitly determine the Hilbert series for the resulting Artinian reduction.  As corollaries, we obtain the multiplicity and regularity of the original rings. This work extends results easily derived from lattice theory for a subfamily coming from a two-sided ladder to a family where, as we show, lattice theory no longer applies in any obvious way and includes constructive proofs which may be useful in future study of these rings and others. 
\end{abstract}

\maketitle

\section{Introduction}

In recent decades, there has been a growing interest in the investigation of algebraic invariants associated to combinatorial structures.  Toric ideals of graphs (and the associated edge rings), a special case of the classical notion of a toric ideal, have been studied by various authors with regard to invariants such as depth, dimension, projective dimension, regularity, graded Betti numbers, Hilbert series, and multiplicity, usually for particular families of graphs (see for example \cite{KOH,BIERMANN2017,Ferrers,DALI20153862,FKVT20,Galet,GITLER2005107,GREIF20201,HH20,HHKO11,HK14,HMO16,HMT19,mori2020edge,RNRN,TATAKIS20111540}).  We note in Remarks~\ref{graphtaurmk} and~\ref{distinct} that the family we consider does not overlap at all or for large $n$ with those considered in  \cite{Ferrers}, \cite{FKVT20}, \cite{Galet}, and \cite{RNRN}; it is more obviously distinct from other families that have been studied.  We think it fitting to mention that the recent book by Herzog, Hibi, and Ohsugi (\cite{herzog-hibi-ohsugi}) also investigates toric ideals of graphs as well as binomial ideals coming from other combinatorial structures.  

In this work, we consider a family of graphs with iterated subfamilies and develop algebraic properties of the toric rings associated to the family which depend only on the number of vertices (equivalently, the number of edges) in the associated graphs.  In the development of this project, we were particularly inspired by the work of Jennifer Biermann, Augustine O'Keefe, and Adam Van Tuyl in \cite{BIERMANN2017}, where they establish a lower bound for the regularity of the toric ideal of any finite simple graph and an upper bound for the regularity of the toric ideal of a chordal bipartite graph.  Our goal is to construct as ``simple" a family of graphs as possible that still yields interesting toric ideals.  It is our hope that our process and results will lead to further generalizations of properties of toric ideals for other (perhaps broader) families of graphs, or for graphs containing or arising from such graphs.

Herein, we introduce the infinite family $\mathcal{F}$ of chordal bipartite graphs $G_n^t$, where $n$ determines the number of edges and vertices and $t$ determines the structure of the graph, and establish some algebraic properties of the toric rings $R(n,t)$ associated to the graphs $G_n^t$.  The use of bipartite graphs makes each $R(n,t)$ normal and Cohen-Macaulay by \cite{SVV94} and \cite{herzog-hibi-ohsugi}; we use the latter in Section~\ref{resultsprops}.  Our main results prove to be independent of $t$ and depend only on $n$.

In Section \ref{family}, we construct the family $\mathcal{F}$ of graphs $G_n^t$ from a family of ladder-like structures $L_n^t$ so that the toric ideals of the $G_n^t$ are generalized determinantal ideals of the $L_n^t$.  The ladder-like structures associated to a subfamily $\mathcal{F}_1\subset\mathcal{F}$, introduced in Example \ref{mexmp2}, are in fact two-sided ladders (for large $n$), so that the family of rings $R(n,t)$ is a generalization of the family of ladder determinantal rings coming from $\mathcal{F}_1$.  While the rings arising from $\mathcal{F}_1$ come from a distributive lattice and have easily derived properties (see for example \cite{herzog-hibi-ohsugi}), we show that the rings associated to $\mathcal{F}$ do not naturally arise from any lattice in general, and merit closer study.

In Section \ref{resultsprops}, we establish some algebraic properties of the $R(n,t)$, particularly Krull dimension, projective dimension, multiplicity, and regularity.  To do so, we prove that the determinantal generators of the defining ideal $I_{G_n^t}$ are a Gr\"obner basis (it follows immediately from  \cite{herzog-hibi-ohsugi} that $R(n,t)$ is Koszul) and work with the initial ideal $\init I_{G_n^t}$.  We also develop a system of parameters $\overline{X_n}$ that allows us to work with Artinian reductions in part of our treatment, and their Hilbert series.  

Our first result establishes the Krull dimension of the toric ring $R(n,t)=S(n)/I_{G_n^t}$, where the ring  $S(n)=k[x_0,x_2,x_3,\ldots,x_{2n+3},x_{2n+4}]$ is the polynomial ring over the edges of $G_n^t$ and $I_{G_n^t}$ is the toric ideal of $G_n^t$.

\begin{thm}[Theorem \ref{dimension}]

The dimension of $R(n,t)$ is 
\[
\dim R(n,t)= n+3.
\]
\end{thm}

\noindent As a corollary, since $R(n,t)$ comes from a bipartite graph and is hence Cohen-Macaulay (Corollary~\ref{CM}), we obtain the projective dimension of $R(n,t)$.

\begin{cor}[Corollary \ref{pdimension}]

The projective dimension of $R(n,t)$ over $S(n)$ is 
\[
\pdim_{S(n)} R(n,t)= n+1.
\]

\end{cor}

We then develop a linear system of parameters for $R(n,t)$, using differences of elements on antidiagonals of the ladder-like structure $L_n^t$.

\begin{prop}[Proposition \ref{sopedgering}]

Let $R(n,t)=S(n)/I_{G_n^t}$.  Then the image of
\begin{equation*}
    X_n=x_{0},x_{2}-x_{3},x_{4}-x_{5},\ldots,x_{2n}-x_{2n+1},x_{2n+2}-x_{2n+3},x_{2n+4}
\end{equation*}
in $R(n,t)$ is a system of parameters for $R(n,t)$.  
\end{prop}
\noindent Since $R(n,t)$ is Cohen-Macaulay, the linear system of parameters above is a regular sequence (Corollary \ref{xt}).  

With the aim of obtaining the multiplicity and regularity of $R(n,t)$, we form an Artinian quotient of $R(n,t)$ by the regular sequence above and call it $\widehat{R(n,t)}$. We note that $\widehat{R(n,t)}$ does not denote the completion, and explain the choice of notation in Definition~\ref{widehatnotation}.  

Using a convenient vector space basis for $\widehat{R(n,t)}$ established in Lemma~\ref{uniquerep}, we show the coefficients of the Hilbert series for $\widehat{R(n,t)}$.  

\begin{thm}[Theorem \ref{Hilbert Series mod reg seq}]

If $R(n,t)=S(n)/I_{G_n^t}$ and $\widehat{R(n,t)}\cong R(n,t)/(\overline{X_n})$, we have 
\begin{equation*}
    {\displaystyle \dim_k(\widehat{R(n,t)})_{i}=\begin{cases}
1 & i=0\\
{\displaystyle \frac{1}{i!}\prod_{j=1}^{i}(n+j-2(i-1))} & i\geq 1.
\end{cases}}
\end{equation*}
In particular, $\dim_k(\widehat{R(n,t)})_{i}=0$ when $i>n/2+1$.

\end{thm}

\noindent As a corollary, we obtain the regularity of $R(n,t)$, which is equal to the top nonzero degree of $\widehat{R(n,t)}$.  

\begin{cor}[Corollary \ref{regcor}]
    For $G_n^t\in\mathcal{F}$,
    \begin{equation*}
    \reg R(n,t)=\left\lfloor n/2 \right \rfloor+1.
\end{equation*}
\end{cor}

\noindent We include an alternate graph-theoretic proof of the result above at the end of this work.  Beginning with an upper bound from \cite{BIERMANN2017} (or equivalently for our purposes, one from \cite{HH20}) and then identifying the initial ideal $\init I_{G_n^t}$ with the edge ideal of a graph, we use results from \cite{conca2018squarefree} (allowing us to use $\init I_{G_n^t}$ instead of $I_{G_n^t}$) and then \cite{Tai-Ha-Van-Tuyl-2008} for a lower bound which agrees with our upper bound.

From a recursion established in Lemma~\ref{recursion}, we go on to prove a Fibonacci relationship between the lengths of the Artinian rings $\widehat{R(n,t)}$ in Proposition~\ref{Fibonacci}, and obtain the multiplicity of $R(n,t)$ as a corollary.  In the following, we drop $t$ for convenience.

\begin{cor}[Corollary \ref{multcor}]
    For $n\geq 2$, there is an equality of multiplicities \[
    e(R(n))=e(R(n-1))+e(R(n-2)).\]  In particular, 
    \[e(R(n))=F\left(n+3\right)=\frac{(1+\sqrt{5})^{n+3}-(1-\sqrt{5})^{n+3}}{2^{n+3}\sqrt{5}}.
    \]
\end{cor}

For more background, detail, and motivation, we refer the reader to \cite{LauraThesis}, but note that different notation and indexing conventions have been employed in this work.  Throughout, $k$ is a field.

\vspace{12pt}

\noindent \textbf{Acknowledgements.} Macaulay2 \cite{M2} was used for computation and hypothesis formation. We would like to thank Syracuse University for its support and hospitality and Claudia Miller for her valuable input on the original project in \cite{LauraThesis} and this condensed version.  We also acknowledge the partial support of an NSF grant.

\section{The Family of Toric Rings} 
\label{family}

In the following, we define a family of toric rings $R(n,t)$ coming from an iterative chordal bipartite family of graphs, $\mathcal{F}$.  We show that although one subfamily of these rings comes from join-meet ideals of a (distributive) lattice and has some easily derived algebraic invariants, this is not true in general. The reader may find the definition of the toric ideal of a graph in Section~\ref{toricringsforF}, when it becomes relevant to the discussion.  We recall for the reader that a \textit{chordal bipartite} graph is a bipartite graph in which every cycle of length greater than or equal to six has a chord. 

\subsection{The Family \texorpdfstring{$\mathcal{F}$}{F} of Graphs}

Below, we define the family $\mathcal{F}$ of chordal bipartite graphs iteratively from a family of ladder-like structures $L_n^t$.  We note that the quantities involved in the following definition follow patterns as follows: 

\begin{center}
\begin{tabular}{ |c|c|c| } 
 \hline
 $n$ & $\left \lfloor{n/2}\right \rfloor+2$ & $\left \lceil{n/2}\right \rceil+2$\\ 
  \hline
 0 & 2 & 2\\ 
 1 & 2 & 3\\ 
 2 & 3 & 3\\ 
 3 & 3 & 4\\ 
 \vdots & \vdots & \vdots\\ 
 \hline
\end{tabular}
\end{center}

\begin{defn}\label{Ln}
For each $n\geq 0$ and each $t\in \F_2^{n+1}$, we construct a ladder-like structure $L_n^t$ with $(\left \lfloor{n/2}\right \rfloor+2)$ rows and $(\left \lceil{n/2}\right \rceil+2)$ columns and with nonzero entries in the set $\{x_0,x_2,x_3,\ldots,x_{2n+4}\}$.  To do so, we use the notation $\widehat{t}\in\F_2^{n}$ for the first $n$ entries of $t$, that is, all except the last entry. The construction is as follows, where throughout, indices of entries in $L_n^t$ are strictly increasing from left to right in each row and from top to bottom in each column.  We note that $L_n^t$ does not depend on $t$ for $n<2$, but does for $n\geq 2$.
\begin{itemize}
    \item For $n=0$, the ladder-like structure $L_0^0=L_0^1$ is 
    \[
    \begin{matrix}
    x_0 & x_2 \\
    x_3 & x_4 
    \end{matrix}
    \]
    \item For $n=1$, to create $L_1^t$ (regardless of what $t$ is in $\F_2^2$), we add another column with the entries $x_5$ and $x_6$ to the right of $L_0^{\widehat{t}}$ to obtain
    \[
    \begin{matrix}
    x_0 & x_2 & x_5\\
    x_3 & x_4 & x_6
    \end{matrix}
    \]
    \item For $2\leq n\equiv 0 \mod 2 (\equiv 1 \mod 2)$, to create $L_n^t$, we add another row (column) with the entries $x_{2n+3},x_{2n+4}$ below (to the right of) $L_{n-1}^{\widehat{t}}$ in the following way:
        \begin{itemize}
            \item [$\circ$] The entry $x_{2n+4}$ is in the ultimate row and column, row $\left \lfloor{n/2}\right \rfloor+2$ and column $\left \lceil{n/2}\right \rceil+2$.
            \item [$\circ$] The entry $x_{2n+3}$ is in the new row (column) in a position directly below (to the right of) another nonzero entry in $L_n^t$.
            \begin{itemize}
                \item If the last entry of $t$ is $0$, $x_{2n+3}$ is directly beneath (to the right of) the first nonzero entry in the previous row (column).
                \item If the last entry of $t$ is $1$, $x_{2n+3}$ is directly beneath (to the right of) the second nonzero entry in the previous row (column).
            \end{itemize}
        \end{itemize}
\end{itemize}

\noindent In this way, the entries in $t$ determine the choice at each stage for the placement of $x_{2n+3}$.
\end{defn}

\begin{rmk}\label{construction}
We note a few things about this construction for $n\equiv 0 \mod 2$ ($\equiv 1 \mod 2$), which may be examined in the examples below:
\begin{itemize}
    \item We note that $x_{2n+4}$ is directly beneath (to the right of) $x_{2n+2}$.
    \item We note that the only entries in row $\left \lfloor{n/2}\right \rfloor+1$ (column $\left \lceil{n/2}\right \rceil+1$) of $L_{n-1}^{\widehat{t}}$ are $x_{2n-1}$, $x_{2n}$, and $x_{2n+2}$, so that the choices listed for placement of $x_{2n+3}$ are the only cases.  In particular, $t_{n+1}=0$ if and only if $x_{2n+3}$ is directly beneath (to the right of) $x_{2n-1}$, and $t_{n+1}=1$ if and only if $x_{2n+3}$ is directly beneath (to the right of) $x_{2n}$.
    \item Finally, we note that the only entries in column $\left \lceil{n/2}\right \rceil+2$ (row $\left \lfloor{n/2}\right \rfloor+2$) of $L_{n}^{t}$ are $x_{2n+1}$, $x_{2n+2}$, and $x_{2n+4}$, and that the only entries in row $\left \lfloor{n/2}\right \rfloor+2$ (column $\left \lceil{n/2}\right \rceil+2$) of $L_{n}^{t}$ are $x_{2n+3}$ and $x_{2n+4}$.
\end{itemize}
\end{rmk}

\begin{exmp}\label{mexmp1}
    We have
\[
L_2^{(1,1,1)}=
\begin{matrix}
x_0 & x_2 & x_5\\
x_3 & x_4 & x_6\\
    & x_7 & x_8
\end{matrix}
\hspace{2cm}
L_2^{(0,0,0)}=
\begin{matrix}
x_0 & x_2 & x_5\\
x_3 & x_4 & x_6\\
x_7 &     & x_8
\end{matrix}
\]

In either of the cases above, we could go on to construct $L_3^{\widehat{t}}$ and $L_4^t$ in the following way: For $n=3$, place $x_{10}$ to the right of $x_8$ and place $x_9$ to the right of either $x_5$ or $x_6$, depending whether the last entry of $\widehat{t}$ is $0$ or $1$, respectively.  Then for $n=4$, place $x_{12}$ below $x_{10}$ and place $x_{11}$ below either $x_7$ or $x_8$, depending whether the last entry of $t$ is $0$ or $1$, respectively.

\end{exmp}

\begin{exmp}\label{mexmp2}

In fact, when the entries of $t$ are all ones, we see that $L_n^{(1,1,\ldots,1)}$ has a ladder shape (is a two-sided ladder for $n\geq 3)$, shown below in the case when $2\leq n\equiv 0 \mod 2$:
\[
\begin{matrix}
x_0 & x_2 & x_5 &\\
x_3 & x_4 & x_6 & x_9 &\\
 & x_7 & x_8 & x_{10} & x_{13} &\\
 &  & x_{11} & x_{12} & x_{14} & x_{17} &\\
 &  &  & x_{15} & x_{16} & x_{18} & x_{21} &\\
 &  &  &  & x_{19} & x_{20} & x_{22} & x_{25} &\\
 &  &  &  &  & x_{23} & x_{24} & x_{26} & \ddots &\\
 &  &  &  &  &  & x_{27} & x_{28} & \ddots & x_{2n+1}\\
 &  &  &  &  &  &  & \ddots & \ddots & x_{2n+2}\\
 &  &  &  &  &  &  &        & x_{2n+3} & x_{2n+4}.
\end{matrix}
\]
We denote the subfamily of graphs coming from $t=(1,1,\ldots,1)$ by $\mathcal{F}_1\subset \mathcal{F}$.

When the entries of $t$ are all zeros, $L_n^{(0,0,\ldots,0)}$ has the following structure, shown below in the case when $2\leq n\equiv 0 \mod 2$:
\[\begin{matrix}
x_0 & x_2 & x_5 & x_9 & x_{13} & x_{17} & x_{21} & x_{25} & \cdots & x_{2n+1}\\
x_3 & x_4 & x_6\\
x_7 &  & x_8 & x_{10}\\
x_{11} &  &  & x_{12} & x_{14}\\
x_{15} &  &  &  & x_{16} & x_{18}\\
x_{19} &  &  &  &  & x_{20} & x_{22}\\
x_{23} &  &  &  &  &  & x_{24} & x_{26}\\
x_{27} &  &  &  &  &  &  & x_{28} & \ddots\\
\vdots &  &  &  &  &  &  &  & \ddots & x_{2n+2}\\
x_{2n+3} &  &  &  &  &  &  &  &  & x_{2n+4}.
\end{matrix}
\]

For a more varied example, we have $L_{16}^{(1,0,1,0,1,1,0,0,1,1,1,0,0,0,1,0,0)}$ below:
\[
\begin{matrix}
x_0        & x_2 & x_5 & x_9    &        &        &        &        &        &           \\
x_3        & x_4 & x_6 &        &        &        &        &        &        &           \\
           & x_7 & x_8 & x_{10} & x_{13} & x_{17} &        &        &        &           \\
           &     &x_{11}& x_{12}& x_{14} &        &        &        &        &           \\
           &     &x_{15}&       & x_{16} & x_{18} & x_{21} & x_{25} & x_{29} & x_{33}    \\
           &     &     &        & x_{19} & x_{20} & x_{22} &        &        &           \\
           &     &     &        &        & x_{23} & x_{24} & x_{26} &        &           \\
           &     &     &        &        & x_{27} &        & x_{28} & x_{30} &           \\
           &     &     &        &        &        &        & x_{31} & x_{32} & x_{34}    \\
           &     &     &        &        &        &        & x_{35} &        & x_{36}.
\end{matrix}
\]
\end{exmp}

\begin{defn}\label{graphtaudef}
    If we associate a vertex to each row and each column and an edge to each nonzero entry of $L_n^t$, we have a finite simple connected bipartite graph $G_n^t$. The set $V_r$ of vertices corresponding to rows and the set $V_c$ of vertices corresponding to columns form a bipartition of the vertices of $G_n^t$. We say a graph $G$ is in $\mathcal{F}$ if $G=G_n^t$ for some $n\geq 0$ and some $t\in\F_2^{n+1}$.
\end{defn}
\begin{rmk}\label{graphtaurmk}
We note that by construction $G_n^t$ has no vertices of degree one, since each row and each column of $L_n^t$ has more than one nonzero entry.  This ensures that for large $n$ our family is distinct from that studied in \cite{Ferrers}, since a Ferrers graph with bipartitation $V_1$ and $V_2$ with no vertices of degree one must have at least two vertices in $V_1$ of degree $|V_2|$ and at least two vertices in $V_2$ of degree $|V_1|$, impossible for our graphs when $n\geq 3$, as the reader may verify.  We also use the fact that $G_n^t$ has no vertices of degree one for an alternate proof of the regularity of $R(n,t)$ at the end of this work.
\end{rmk}

\begin{exmp}
    When $n=5$, $G_5^{(1,1,\ldots,1)}\in \mathcal{F}_1$ is 
    
    \tikzset{main node/.style={circle,fill=blue!20,draw,minimum size=20pt,inner sep=0pt},
            }
    
    \begin{center}
    
    \begin{tikzpicture}
    \node[main node] (1) {$r1$};
    \node[main node] (5) [below left = 1.4cm and 1.4cm of 1] {$c1$};
    \node[main node] (6) [below right = 1.4cm and 1.4cm of 1] {$c2$};
    \node[main node] (2) [below right = 1.4cm and 1.4cm of 5] {$r2$};
    \node[main node] (7) [left = 1.4cm of 5] {$c3$};
    \node[main node] (3) [above = 1.4cm of 1] {$r3$};
    \node[main node] (8) [right = 1.4cm of 6] {$c4$};
    \node[main node] (4) [below = 1.4cm of 2] {$r4$};
    \node[main node] (9) [left = 1.4cm of 7] {$c5$};

    \path[draw,thick]
    (1) edge node [below] {$x_0$}    (5)
    (1) edge node [below] {$x_2$}    (6)
    (1) edge node [above] {$x_5$}    (7)
    (2) edge node [above] {$x_3$}    (5)
    (2) edge node [above] {$x_4$}    (6)
    (2) edge node [below] {$x_6$}     (7)
    (2) edge node [left] {$x_9$}     (8)
    (3) edge node [right] {$x_7$}    (6)
    (3) edge node [left] {$x_8$}    (7)
    (3) edge node [right] {$x_{10}$}  (8)
    (3) edge node [left] {$x_{13}$}  (9)
    (4) edge node [right] {$x_{11}$} (7)
    (4) edge node [right] {$x_{12}$}  (8)
    (4) edge node [left] {$x_{14}$} (9)
    ;
\end{tikzpicture}

\end{center}

\end{exmp}

We develop properties of the $L_n^t$ which allow us to show in Section~\ref{edgeringsforF} that certain minors of the $L_n^t$ are generators for the toric rings of the corresponding graphs $G_n^t$.

\begin{defn}
    For this work, a \textit{distinguished minor} of $L_n^t$ is a $2$-minor involving only (nonzero) entries of the ladder-like structure $L_n^t$, coming from a $2\times 2$ subarray of $L_n^t$.
\end{defn}

\begin{prop}\label{det}
    For each $i\geq 1$ and each $f\in\F_2^{i+1}$, the entry $x_{2i+3}$ and the entry $x_{2i+4}$ each appear in exactly two distinguished minors of $L_i^f$.  For $i\equiv 0 \mod 2 (\equiv 1 \mod 2)$, these minors are of the form
    \[
    s_{2i}:=x_{2i+1}x_{2i+3}-x_{j_{2i}}x_{2i+4}
    \]
    coming from the subarray
    \[
    \begin{bmatrix}
        x_{j_{2i}} & x_{2i+1} \\
        x_{2i+3} & x_{2i+4}
    \end{bmatrix}
    \hspace{2cm}
    \left(
    \begin{bmatrix}
        x_{j_{2i}} & x_{2i+3} \\
        x_{2i+1} & x_{2i+4}
    \end{bmatrix}\right)
    \]
    for some $j_{2i}\in\{0,2,3,\ldots,2i-2\}$ and 
    \[
    s_{2i+1}:=x_{2i+2}x_{2i+3}-x_{j_{2i+1}}x_{2i+4}
    \]
    coming from the subarray
    \[
    \begin{bmatrix}
        x_{j_{2i+1}} & x_{2i+2} \\
        x_{2i+3} & x_{2i+4}
    \end{bmatrix}
    \hspace{2cm}
    \left(
    \begin{bmatrix}
        x_{j_{2i+1}} & x_{2i+3} \\
        x_{2i+2} & x_{2i+4}
    \end{bmatrix}
    \right)
    \]
    for some $j_{2i+1} \in\{2i-1,2i\}$, and the only distinguished minor of $L_n^t$ with indices all less than $5$ is $s_1:=x_2x_3-x_0x_4$.
\end{prop}

\begin{proof}
    The last statement is clear by Definition~\ref{Ln}; we prove the remaining statements by induction on $i$.  For $i=1$, we have the distinguished minors $s_2=x_3x_5-x_0x_6$ and $s_3=x_4x_5-x_2x_6$ coming from the subarrays
    \[
    \begin{bmatrix}
        x_0 & x_5 \\
        x_3 & x_6
    \end{bmatrix}
    \]
    and
    \[
    \begin{bmatrix}
        x_2 & x_5 \\
        x_4 & x_6
    \end{bmatrix}
    \]
    where $j_2=0\in\{0\}$ and $j_3=2\in\{1,2\}$, so we have our base case. Now suppose the statement is true for $i$ with $1\leq i<n$, and let $n\equiv 0\mod 2 (\equiv 1 \mod 2)$ and $t\in\F_2^{n+1}$.  
    
    Case 1: If $t_{n+1}=0$, then by Remark~\ref{construction}, $x_{2n+3}$ is in the same column (row) as $x_{2n-1}$.  By induction, we have the distinguished minor $s_{2n-2}=x_{2n-1}x_{2n+1}-x_{j_{2n-2}}x_{2n+2}$ coming from the subarray
    \[
    \begin{bmatrix}
        x_{j_{2n-2}} & x_{2n+1} \\
        x_{2n-1} & x_{2n+2}
    \end{bmatrix}
    \hspace{2cm}
    \left(
    \begin{bmatrix}
        x_{j_{2n-2}} & x_{2n-1} \\
        x_{2n+1} & x_{2n+2}
    \end{bmatrix}
    \right).
    \]
    Then in fact we have a subarray of the form 
    \[
    \begin{bmatrix}
        x_{j_{2n-2}} & x_{2n+1} \\
        x_{2n-1} & x_{2n+2}\\
        x_{2n+3} & x_{2n+4}
    \end{bmatrix}
    \hspace{2cm}
    \left(
    \begin{bmatrix}
        x_{j_{2n-2}} & x_{2n-1} & x_{2n+3}\\
        x_{2n+1} & x_{2n+2} & x_{2n+4}
    \end{bmatrix}
    \right),
    \]
    so that we have the distinguished minors
    \begin{eqnarray*}
    s_{2n}&=&x_{2n+1}x_{2n+3}-x_{j_{2n-2}}x_{2n+4}\\
    s_{2n+1}&=&x_{2n+2}x_{2n+3}-x_{2n-1}x_{2n+4}
    \end{eqnarray*}
    with 
    \[
    j_{2n}=j_{2n-2}\in\{0,2,3,\ldots,2n-4\}\subset\{0,2,3,\ldots,2n-2\}
    \]
    by induction and with 
    \[
    j_{2n+1}=2n-1\in\{2n-1,2n\}.  
    \]
    Since the only entries in row $\left \lfloor{n/2}\right \rfloor+2$ (column $\left \lceil{n/2}\right \rceil+2$) of $L_n^t$ are $x_{2n+3}$ and  $x_{2n+4}$ and since the only entries in column $\left \lceil{n/2}\right \rceil+2$ (row $\left \lfloor{n/2}\right \rfloor+2$) of $L_n^t$ are $x_{2n+1}$, $x_{2n+2}$, and $x_{2n+4}$ by Remark~\ref{construction}, these are the only distinguished minors of $L_n^t$ containing either $x_{2n+3}$ or $x_{2n+4}$, as desired.
    
    Case 2 for $t_{n+1}=1$ is analogous and yields 
    \[
    j_{2n}=j_{2n-1}\in\{2n-3,2n-2\}\subset\{0,2,3,\ldots,2n-2\}
    \]
    and
    \[
    j_{2n+1}=2n\in\{2n-1,2n\}.  
    \]
    
\end{proof}

\begin{defn}
    Define the integers $j_{2i}$, $j_{2i+1}$ for $j_2,\ldots,j_{2n+1}$ as in the statement of Proposition~\ref{det}.  We note in the remark below some properties of the $j_k$.
\end{defn}

\begin{rmk}\label{jk}
    From the proof of Proposition~\ref{det}, we note that $j_2=0$, $j_3=2$, and that for $i\geq 2$, we have the following:  
    \begin{eqnarray*}
    t_{i+1}&=&0\iff j_{2i}=j_{2i-2}\iff j_{2i+1}=2i-1\\
    t_{i+1}&=&1\iff j_{2i}=j_{2i-1}\iff j_{2i+1}=2i.
    \end{eqnarray*}
    For the sake of later proofs, we extend the notion of $j_k$ naturally to $s_1=x_2x_3-x_0x_4$ and say that $j_1=0$, and note the following properties of the $j_k$ for $1 \leq k \leq 2n+1$:
    \begin{itemize}
        \item We have $j_{2i}\in\{j_{2i-2},j_{2i-1}\}$ and $j_{2i}\leq 2i-2$. Indeed, for $i=1$, $j_2=j_1=0$, and for $i\geq 2$, this is clear from the statement above.
        \item We have $j_{2i+1}\in\{2i-1,2i\}$.  Indeed, for $i=0$, $j_1=0\in\{-1,0\}$, for $i=1$, $j_3=2\in\{1,2\}$, and for $i\geq 2$, this follows from the statement above.
        \item The $j_{2i}$ form a non-decreasing sequence.  Indeed, for $i\geq 2$, either $j_{2i}=j_{2i-2}$ or $j_{2i}=j_{2i-1}\geq 2i-3>2i-4\geq j_{2i-2}$.  
    \end{itemize}
\end{rmk}

\begin{rmk}\label{subarray}
   We also note from the proof above that the following is a subarray of $L_n^t$ for all $i\equiv 0 \mod 2 (\equiv 1 \mod 2)$ such that $1\leq i\leq n$, which we use in the proof of the proposition below:
    \[
    \begin{bmatrix}
        x_{j_{2i}} & x_{2i+1} \\
        x_{j_{2i+1}} & x_{2i+2}\\
        x_{2i+3} & x_{2i+4}
    \end{bmatrix}
    \hspace{2cm}
    \left(
    \begin{bmatrix}
        x_{j_{2i}} & x_{j_{2i+1}} & x_{2i+3} \\
        x_{2i+1} & x_{2i+2} & x_{2i+4}
    \end{bmatrix}
    \right)
    \]
\end{rmk}

\begin{prop}\label{chordal}

For $n\geq 0$, each graph $G_n^t\in \mathcal{F}$ is chordal bipartite with vertex bipartition $V_r\cup V_c$ of cardinalities
\begin{eqnarray*}
    |V_r|=\left \lfloor \frac{n}{2} \right \rfloor +2\\
    |V_c|=\left \lceil \frac{n}{2} \right \rceil +2.
\end{eqnarray*}

\end{prop}

\begin{proof}
    We already know by Definition~\ref{graphtaudef} that every graph $G_n^t$ is bipartite for $n\geq 0$, with the bipartition above coming from the rows and columns of $L_n^t$.  The cardinalities of the vertex sets follow from Definition~\ref{construction}. We prove the chordal bipartite property by induction on $n$.  It is clear for $i=0$ and $i=1$ that $G_i^f$ is chordal bipartite for $f\in \F_2^{i+1}$, since these graphs have fewer than six vertices.  Now suppose $G_{i}^f$ is chordal bipartite for $i$ with $1\leq i<n\equiv 0 \mod 2 (\equiv 1 \mod 2)$, and consider $G_n^t$ for $t\in \F_2^{n+1}$.  
    We know that the following array (or its transpose) is a subarray of $L_n^t$ by Remark~\ref{subarray}, and we include for reference the corresponding subgraph of $G_n^t$ with vertices labeled by row and column.
    \[
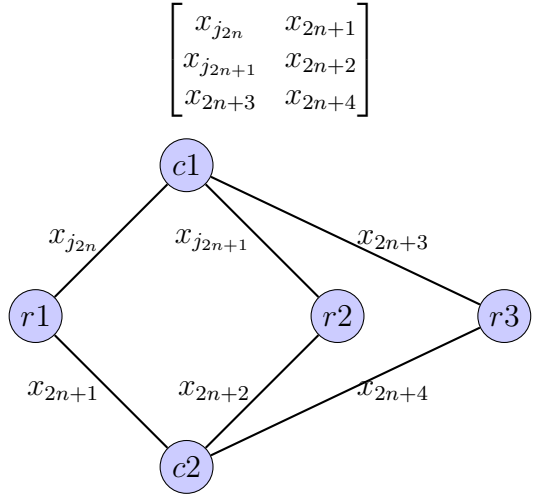
    \begin{bmatrix}
        x_{j_{2n}} & x_{2n+1} \\
        x_{j_{2n+1}} & x_{2n+2}\\
        x_{2n+3} & x_{2n+4}
    \end{bmatrix}
    \]
    
    \begin{center}

    \tikzset{main node/.style={circle,fill=blue!20,draw,minimum size=20pt,inner sep=0pt},
            }
    
    \begin{tikzpicture}
    
    \node[main node] (4) {$c1$};
    \node[main node] (1) [below left = 1.5cm and 1.5cm of 4] {$r1$};
    \node[main node] (2) [below right = 1.5cm and 1.5cm of 4] {$r2$};
    \node[main node] (5) [below right = 1.5cm and 1.5cm of 1] {$c2$};
    \node[main node] (3) [right = 1.5cm of 2] {$r3$};
    
    \path[draw,thick]
    (1) edge node [left] {$x_{j_{2n}}$}    (4)
    (1) edge node [left] {$x_{2n+1}$}    (5)
    (2) edge node [left] {$x_{j_{2n+1}}$}(4)
    (2) edge node [left] {$x_{2n+2}$}    (5)
    (3) edge node [right] {$x_{2n+3}$}     (4)
    (3) edge node [right] {$x_{2n+4}$}    (5)
    ;
\end{tikzpicture}

\end{center}
    We know the only difference between $G_n^t$ and $G_{n-1}^{\widehat{t}}$ is one vertex $r_3$ corresponding to row $\left \lfloor{n/2}\right \rfloor+2$ (column $\left \lceil{n/2}\right \rceil+2$) and two edges $x_{2n+3}=\{r_3,c_1\}$ and $x_{2n+4}=\{r_3,c_2\}$, where $c_1$ corresponds to the column (row) containing $x_{2n+3}$ and $c_2$ corresponds to column $\left \lceil{n/2}\right \rceil+2$ (row $\left \lfloor{n/2}\right \rfloor+2$).  By Remark~\ref{construction}, $\deg r_3=2$, since the only entries in row $\left \lfloor{n/2}\right \rfloor+2$ (column $\left \lceil{n/2}\right \rceil+2$) are $x_{2n+3}$ and $x_{2n+4}$.  Then any even cycle containing $r_3$ must also contain $x_{2n+3}$ and $x_{2n+4}$.  Similarly, by the same remark, the only other edges with endpoint $c_2$ are $x_{2n+1}$ and $x_{2n+2}$, the entries added to make $L_{n-1}^{\widehat{t}}$, so we know that any even cycle containing $x_{2n+4}$ and $x_{2n+3}$ must contain either $x_{2n+1}$ or $x_{2n+2}$.  We see that any even cycle containing $r_3$ and $x_{2n+1}$ is either a 4-cycle or has $x_{j_{2n}}$ as a chord, and any even cycle containing $r_3$ and $x_{2n+2}$ is either a 4-cycle or has $x_{j_{2n+1}}$ as a chord.  Thus every graph $G_n^t$ is chordal bipartite for $n\geq 0$, with the bipartition above.
\end{proof}

\begin{rmk}\label{distinct}

We note that the previous proposition ensures that our graphs are distinct from those studied in \cite{FKVT20}, \cite{Galet}, and \cite{RNRN}, which are not chordal bipartite except for the first family in \cite{FKVT20}, in which every four-cycle shares exactly one edge with every other four-cycle (also distinct from our family except for the trivial case with only one four-cycle, corresponding to $G_0^t$).  

\end{rmk}

\subsection{Toric Rings for \texorpdfstring{$\mathcal{F}$}{F}}\label{edgeringsforF}
\label{toricringsforF}

In this section, we develop the toric ring $R(n,t)$ for each of the chordal bipartite graphs $G_n^t$ in the family $\mathcal{F}$.  We first show that the toric ideal $I_{G_n^t}$ of the graph $G_n^t$ is the same as the ideal $I(n,t)$ generated by the distinguished minors of $L_n^t$.  We then demonstrate that for some $n$ and $t$, these ideals do not arise from the join-meet ideals of lattices in a natural way, so that results in lattice theory do not apply to the general family $\mathcal{F}$ in an obvious way.

We first define the toric ideal of a graph.  For any graph $G$ with vertex set $V$ and edge set $E$, there is a natural map $\pi:k[E]\to k[V]$ taking an edge to the product of its endpoints.  The polynomial subring in $k[V]$ generated by the images of the edges under the map $\pi$ is denoted $k[G]$ and is called the \textit{edge ring of $G$}.  The kernel of $\pi$ is denoted $I_G$ and is called the \textit{toric ideal of $G$}; the ring $k[G]$ is isomorphic to $k[E]/I_G$.  In this work, we consider the toric ring $k[E]/I_G$ and the toric ideal $I_G$ for our particular graphs.  It is known in general that $I_G$ is generated by binomial expressions coming from even closed walks in $G$ \cite[Prop 3.1]{V95} and that the toric ideal of a chordal bipartite graph is generated by quadratic binomials coming from the $4$-cycles of $G$ (see \cite[Th 1.2]{HO99}).

Let $S(n)=k[x_{0},x_{2},x_{3},\ldots,x_{2n+4}]$ be the polynomial ring in the edges of $G_n^t$. The edge ring for $G_n^t\in\mathcal{F}$ is denoted by $k[G_n^t]$ and is isomorphic to the toric ring 

\begin{equation*}
    R(n,t):=\frac{S(n)}{I_{G_n^t}},
\end{equation*}
where $I_{G_n^t}$ is the toric ideal of $G_n^t$ \cite[5.3]{herzog-hibi-ohsugi}.  
For the general construction of a toric ideal, we refer the reader to \cite[Ch 4]{Peeva2011}. Our goal is to show that the toric ideal $I_{G_n^t}$ of $G_n^t$ is equal to 
\begin{equation*}
    I(n,t)=(\{\text{distinguished minors of $L_n^t$}\}).
\end{equation*}

\begin{prop}\label{edgerings}

Let $S(n)=k[x_{0},x_{2},x_{3},\ldots,x_{2n+4}]$.  For $G_n^t\in\mathcal{F}$, we have

\begin{equation*}
    R(n,t)=\frac{S(n)}{I(n,t)},
\end{equation*}
where 
\begin{equation*}
    I(n,t)=(\{\text{distinguished minors of $L_n^t$}\}).
\end{equation*}
\end{prop}

\begin{proof}
    To prove this, we need only show that $I(n,t)$ is the toric ideal $I_{G_n^t}$ of the graph $G_n^t$.  By Definition~\ref{graphtaudef}, it is clear that the distinguished minors of $L_n^t$ are in $I_{G_n^t}$, corresponding to the $4$-cycles of $G_n^t$.  Since $G$ is chordal bipartite by Proposition~\ref{chordal}, these are the only generators of $I_{G_n^t}$ \cite[Cor 5.15]{herzog-hibi-ohsugi}.
\end{proof}

\begin{cor}\label{CM}
    The rings $R(n,t)$ are normal Cohen-Macaulay rings.
\end{cor}

\begin{proof}
    By Definition~\ref{graphtaudef} and Proposition~\ref{edgerings}, the ring $R(n,t)$ is the toric ring of a finite simple connected bipartite graph, and hence by Corollary 5.26 in \cite{herzog-hibi-ohsugi}, $R(n,t)$ is Cohen-Macaulay for each $n$ and $t$.  The fact that each $R(n,t)$ is normal follows from \cite[Th 5.9, 7.1]{SVV94}.
\end{proof}

Because we know the distinguished minors of $L_n^t$, we are now able to characterize the generators for the toric ideal $R(n,t)$ of $G_n^t$.  

\begin{rmk}\label{ti}
    By Proposition~\ref{det}, the generators $s_1,\ldots,s_{2n+1}$ for $I_{G_n^t}$ may be summarized as follows.  For integers $i$ such that $1\leq i \leq n$, set
    \begin{eqnarray*}
    s_{1}&=&x_{2}x_{3}-x_{j_{1}}x_{4}\\
    s_{2i}&=&x_{2i+1}x_{2i+3}-x_{j_{2i}}x_{2i+4}\\
    s_{2i+1}&=&x_{2i+2}x_{2i+3}-x_{j_{2i+1}}x_{2i+4},
    \end{eqnarray*}
    where the nonnegative integers $j_{k}$ are as in Remark~\ref{jk}, that is, $j_1=j_2=0$, $j_3=2$, and for $i\geq 2$, we have
    \begin{eqnarray*}
    t_{i+1}&=&0\iff j_{2i}=j_{2i-2}\iff j_{2i+1}=2i-1\\
    t_{i+1}&=&1\iff j_{2i}=j_{2i-1}\iff j_{2i+1}=2i.
    \end{eqnarray*}
    We note that the number of generators depends on $n$ and that the $j_k$ depend on $t$, but we may ignore dependence on $t$ when working with general $j_k$.
    We sometimes call $s_1,\ldots,s_{2n+1}$ the \textit{standard generators} of $I_{G_n^t}$, and show in Section~\ref{distinction} that for certain $n$ and $t$, they are not equal to the usual generators for the join-meet ideal of any lattice $D$.
\end{rmk}

\begin{exmp}
    We consider the toric ideal of a graph in $\mathcal{F}_1$.  For $n=5$ and $t=(1,1,\ldots,1)$, by Remark~\ref{jk} we have $j_1=j_2=0$, $j_3=2$, $j_{2i}=j_{2i-1}$ and $j_{2i+1}=2i$ for $i\geq 2$, so that
    \[
    R(5,(1,1,\ldots,1))=\frac{k[x_0,x_2,\ldots,x_{14}]}{I_{G_5^{(1,1,\ldots,1)}}},
    \]
    where $I_{G_5^{(1,1,\ldots,1)}}$ is generated by the distinguished minors of $L_5^{(1,1,\ldots,1)}$:
    \begin{align*}
    s_1 &= x_2x_3-x_0x_4 & s_2 &= x_3x_5-x_0x_6\\
    s_3 &= x_4x_5-x_2x_6 & s_4 &= x_5x_7-x_2x_8\\
    s_5 &= x_6x_7-x_4x_8 & s_6 &= x_7x_9-x_4x_{10}\\
    s_7 &= x_8x_9-x_6x_{10} & s_8 &= x_9x_{11}-x_6x_{12}\\
    s_9 &= x_{10}x_{11}-x_8x_{12} & s_{10} &=x_{11}x_{13}-x_8x_{14}\\
    s_{11} &= x_{12}x_{13}-x_{10}x_{14}.
    \end{align*}

\end{exmp}

\subsection{Distinction From Join-Meet Ideals of Lattices}\label{distinction}

We saw in Example~\ref{mexmp2} and Proposition~\ref{edgerings} that if $G_n^t\in\mathcal{F}_1\subset\mathcal{F}$, then $I_{G_n^t}$ is a ladder determinantal ideal for $n\geq 2$.  It is known that a ladder determinantal ideal is equal to the join-meet ideal of a (distributive) lattice (indeed, with a natural partial ordering which decreases along rows and columns of $L_n^t$ we obtain such a lattice).  Some algebraic information such as regularity and projective dimension may be easily derived for some join-meet ideals of distributive lattices (see, for example, Chapter 6 of \cite{herzog-hibi-ohsugi}). We spend some time in this section establishing that not all rings $R(n,t)\in \mathcal{F}$ arise from a lattice in a natural way (see Remark~\ref{natural}), and so there does not seem to be any obvious way to obtain our results in Section~\ref{resultsprops} from the literature on join-meet ideals of distributive lattices or on ladder determinantal ideals.  The results in Section~\ref{resultsprops} may be viewed as an extension of what may already be derived for the family $\mathcal{F}_1$ from the existing literature.

The following five lemmas serve to provide machinery to show that there is at least one ring in the family $\mathcal{F}$, namely $R(5,(1,1,1,1,1,0))$, whose toric ideal does not come from a lattice on the set $\{x_0,x_2,\ldots,x_{14}\}$ in any obvious way.  That is, we show that the standard generators of $I_{G_{5}^t}$, the $s_k$ from Remark~\ref{ti}, are not equal to the standard generators (see Definition~\ref{standard}) for any lattice $D$ on $\{x_0,x_2,\ldots,x_{14}\}$.  

Before we begin, we introduce some definitions and notation that we use extensively throughout.

\begin{defn}\label{standard}
    The join-meet ideal of a lattice is defined from the join (least upper bound) $x\vee y$ and meet (greatest lower bound) $x\wedge y$ of each pair of incomparable elements $x,y\in L$.  In this work, a \textit{standard generator} of the join-meet ideal of a lattice $D$ is a nonzero element of one of the following four forms: 
    \begin{eqnarray*}
        x_ax_b-(x_a\vee x_b)(x_a\wedge x_b) &=& x_ax_b-(x_a\wedge x_b)(x_a\vee x_b)\\
        (x_a\vee x_b)(x_a\wedge x_b)-x_ax_b &=&
        (x_a\wedge x_b)(x_a\vee x_b)-x_ax_b
    \end{eqnarray*}
    for $x_a,x_b\in L$.  We sometimes refer to such an element as a \textit{standard generator of $D$} (the join-meet ideal is defined by analogous generators in the literature, though sometimes $a,b\in L$ instead of $x_a$ and $x_b$). We note that for a standard generator, the pair $\{x_a,x_b\}$ is an incomparable pair, and the pair $\{(x_a\vee x_b),(x_a\wedge x_b)\}$ is a comparable pair.  
\end{defn}

Though we are in a commutative ring, we provide all possible orderings for factors within the terms of a standard generator to emphasize that either factor of the monomial
\[
(x_a\vee x_b)(x_a\wedge x_b)=(x_a\wedge x_b)(x_a\vee x_b)
\]
may be the join or the meet of $x_a$ and $x_b$.

\begin{rmk}\label{natural}
   We give an explanation for why it makes sense to focus only on the standard generators of a join-meet ideal.  We recall that the standard generators $s_k$ for $I_{G_n^t}$  from Remark~\ref{ti} come from distinct $2\times 2$ arrays within the ladder-like structure $L_n^t$ and recognize that either monomial of $s_k$ determines its $2 \times 2$ array.  Then an element of the form $ab-cd$ in  $I_{G_{5}^t}$ with $a,b,c,d\in \{x_0,x_2,x_3,\ldots,x_{14}\}$ must be equal to $\pm s_i$ for some $i$, since a nontrivial sum of $s_k$ with coefficients in $\{-1,1\}$ either has more than two terms or is equal to $s_i$ for some $i$, and other coefficients would be extraneous.  Then any generating set for $I_{G_{5}^t}$ where each element has the form $ab-cd$ in  $I_{G_{5}^t}$ with $a,b,c,d\in \{x_0,x_2,x_3,\ldots,x_{14}\}$ must consist of all the $s_k$ (up to sign).  We conclude that it is natural to check whether the $s_k$ are standard generators of a lattice $D$, instead of non-standard generators. 
\end{rmk}

\begin{defn}\label{F}
    Given a standard generator $s=uz-wv$ of a lattice $D$, where $u,v,w,z\in D$, let $F_s\in \F_2$ be defined as follows: 
    
    \begin{itemize}
        \item If $F_s=0$, the elements in the first (positive) monomial of $s$ are not comparable in $D$ (so the elements in the second (negative) monomial of $s$ are comparable in $D$). 
        
        \item If $F_s=1$, the elements in the negative monomial of $s$ are not comparable in $D$ (so the elements in the positive monomial of $s$ are comparable in $D$).
    \end{itemize}
    For a given list $s_1,s_2,\ldots,s_m$ of standard generators of a lattice $D$, we use 
    \begin{equation*}
        F=(F_1,\ldots,F_{m})\in\F_2^{m},
    \end{equation*}
    where $F_j=F_{s_j}$, to encode the comparability of the variables in these generators.
\end{defn}
\noindent We note that exactly one of $F_j=0$ or $F_j=1$ happens for each $j$; we are merely encoding which monomial in $s_j$ corresponds to $x_ax_b$, and which to $(x_a\vee x_b)(x_a\wedge x_b)=(x_a\wedge x_b)(x_a\vee x_b)$.  

\begin{notation}
    We use the notation $u>\{w,v\}$ if $u>w$ and $u>v$ in a lattice $D$, and $\{w,v\}>z$ if $w>z$ and $v>z$ in $D$.
\end{notation}

In the first lemma, we begin by showing what restrictions we must have on a lattice whose join-meet ideal contains the 2-minors of the following array as standard generators:

\vspace{12pt}

\begin{center}
$\begin{matrix}
a & b & e\\
c & d & f
\end{matrix}$
\end{center}

\vspace{12pt}

\begin{lemma}\label{3rel}
    Suppose 
    \begin{eqnarray*}
    s_1&=&bc-ad\\
    s_2&=&ce-af\\
    s_3&=&de-bf
    \end{eqnarray*}
    are standard generators of a lattice $D$. Let $F\in\F_2^3$ be defined for these three elements as in Definition~\ref{F}.  Then up to relabeling of variables,
    \begin{equation*}
        F\in\{\{0,0,0\},\{0,0,1\},\{0,1,1\}\}.
    \end{equation*}
\end{lemma}

\begin{proof}
    We first note that some of the cases we consider are equivalent.  If we relabel variables according to the permutation $(ac)(bd)(ef)$, we see that 
    \begin{equation*}
        F=\{i,j,k\}\equiv\{1-i,1-j,1-k\}.
    \end{equation*}  
    This limits the cases we need to consider to
    \begin{equation*}
        F\in\{\{0,0,0\},\{0,0,1\},\{0,1,0\},\{0,1,1\}\}.
    \end{equation*}
    That is, we only need to show that the case $F=\{0,1,0\}$ is impossible.
    
    Let $F_1=0$.  Then without loss of generality, up to reversing the order in the lattice (which does not affect the join-meet ideal), we have $a>\{b,c\}>d$.  If $F_2=1$,  we have $e>\{a,f\}>c$, so $e>\{b,f\}>d$ and hence $F_3=1$.  We conclude that the case $F=\{0,1,0\}$ is impossible.
\end{proof}

In the second lemma, we show what restrictions we must have on a lattice whose join-meet ideal contains the 2-minors of the ladder 

\vspace{12pt}

\begin{center}
$\begin{matrix}
a & b & e\\
c & d & f\\
 & g & h\\
\end{matrix}$
\end{center}

\noindent as standard generators, and which meets certain comparability conditions.

\vspace{12pt}

\begin{lemma}\label{5rel}
    Suppose 
    \begin{eqnarray*}
    s_1&=&bc-ad\\
    s_2&=&ce-af\\
    s_3&=&de-bf\\
    s_4&=&eg-bh\\
    s_5&=&fg-dh
    \end{eqnarray*}
    are standard generators of a lattice $D$, and that $\{a,g\}$,$\{a,h\}$,$\{c,g\}$, and $\{c,h\}$ are comparable pairs in $D$. Let $F\in \F_2^{5}$ be defined for these five elements as in~\ref{F}.  Then up to relabeling of variables, $F=\{0,0,0,0,0\}$.
\end{lemma}

\begin{proof}
     We first note that with natural relabeling, both $\{s_1,s_2,s_3\}$ and $\{s_3,s_4,s_5\}$ satisfy the hypotheses of Lemma~\ref{3rel}, so if we let $F$ be defined as in Definition~\ref{F}, this limits the cases we need to consider to 5-tuples whose first three elements and whose last three elements satisfy the conclusion of Lemma~\ref{3rel}. We note that some of the cases we consider are equivalent.  If we relabel variables according to the permutation $(ac)(bd)(ef)$, we see that $F=\{i,j,k,l,m\}\equiv\{1-i,1-j,1-k,m,l\}$, and if we relabel the variables according to the permutation $(be)(df)(gh)$, we have $F=\{i,j,k,l,m\}\equiv\{j,i,1-k,1-l,1-m\}$.  The permutation $(ah)(cg)(bf)$ yields $F=\{i,j,k,l,m\}\equiv\{m,l,k,j,i\}$.  Then by Lemma~\ref{3rel} we have the eighteen cases 
     
     \pagebreak
     
     \begin{equation*}
         \{0,0,0,0,0\}\equiv \{1,1,1,0,0\}\equiv \{1,1,0,1,1\}\equiv \{0,0,1,1,1\}
     \end{equation*}
     \begin{eqnarray*}
         \{0,0,0,0,1\}\equiv \{1,1,1,1,0\}\equiv \{1,1,0,0,1\}\equiv \{0,0,1,1,0\}&\equiv& \{0,1,1,0,0\}\\
         &\equiv& \{1,0,0,0,0\}\\
         &\equiv& \{0,1,1,1,1\}\\
         &\equiv& \{1,0,0,1,1\}
     \end{eqnarray*}
     \begin{equation*}
         \{0,0,0,1,1\}\equiv \{1,1,1,1,1\}\equiv \{1,1,0,0,0\}\equiv \{0,0,1,0,0\}
     \end{equation*}
     \begin{equation*}
         \{0,1,1,1,0\}\equiv \{1,0,0,0,1\}
     \end{equation*}
     
    \vspace{12pt}
     
    \noindent Case 1: $F=\{0,0,0,0,1\}$.  Since $F_1=0$, without loss of generality (reversing the order on the entire lattice if needed) we have $a>\{b,c\}>d$.  Then $F_2=F_3=F_4=0$ and $F_5=1$, with the ordering chosen, yield
     \begin{align*}
        &a>\{c,e\}>f\\
        &b>\{d,e\}>f\\
        &b>\{e,g\}>h\\
        &g>\{d,h\}>f.
     \end{align*}
     If $c>g$, then $a>\{b,c\}>g>d$, but then $bc-ad$ is not a standard generator of $D$, and this is a contradiction.  If $c<g$, then $c<g<b$ so that both $\{b,c\}$ and $\{a,d\}$ from $s_1$ are comparable pairs, but this is a contradiction.
     We conclude that the case $\{0,0,0,0,1\}$ is impossible.
     
    \vspace{12pt}
    
    Because of the comparability of the pair $\{c,g\}$, the case $F=\{0,1,1,1,0\}$ forces comparability of $\{f,g\}$ or $\{b,c\}$ and hence yields a contradiction, as the reader may verify.  In the case $F=\{0,0,0,1,1\}$, the comparability of $\{c,g\}$ forces either the comparability of $\{b,c\}$ (a contradiction), or $d<\{b,c\}<a<g$, up to reversing the order in the lattice, since $s_1$ is a standard generator.  Since $\{a,h\}$ is a comparable pair, it immediately follows that either $\{b,h\}$ is comparable (a contradiction) or $e<\{b,h\}<a<g$, which is a contradiction since $s_4$ is a standard generator, as the reader may verify.
    These cases are compatible with the given relabelings and thus conclude our proof.  
\end{proof}

In the third lemma, we show what restrictions we must have on a lattice whose join-meet ideal contains the 2-minors of the ladder 

\vspace{12pt}

\begin{center}
$\begin{matrix}
a & b & e & \\
c & d & f & i\\
 & g & h & j
\end{matrix}$
\end{center}

\noindent as standard generators and which meets certain comparability conditions.

\vspace{12pt}

\begin{lemma}\label{7rel}
    Suppose 
    \begin{eqnarray*}
    s_1&=&bc-ad\\
    s_2&=&ce-af\\
    s_3&=&de-bf\\
    s_4&=&eg-bh\\
    s_5&=&fg-dh\\
    s_6&=&gi-dj\\
    s_7&=&hi-fj
    \end{eqnarray*}
    are standard generators of a lattice $D$, and that $\{a,g\}$, $\{a,h\}$, $\{c,g\}$, $\{c,h\}$, $\{b,i\}$, $\{b,j\}$, $\{e,i\}$, and $\{e,j\}$ are comparable pairs in $D$. Let $F\in \F_2^{7}$ be defined for these seven elements as in Definition~\ref{F}.  Then up to relabeling of variables, $F=\{0,0,0,0,0,0,0\}$.
\end{lemma}

\begin{proof}
     We first note that with natural relabeling, both $\{s_1,s_2,s_3,s_4,s_5\}$ and $\{s_3,s_4,s_5,s_6,s_7\}$ satisfy the hypotheses of Lemma~\ref{5rel}, so if we let $F$ be defined as in Definition~\ref{F}, this limits the cases we need to consider to 7-tuples whose first five elements and whose last five elements satisfy the conclusion of Lemma~\ref{5rel}.  The only possible cases are $\{0,0,0,0,0,0,0\}$ and  $\{0,0,1,1,1,0,0\}$.  If we relabel variables according to the permutation $(be)(df)(gh)$, we see that $F=\{i,j,k,l,m,n,o\}\equiv\{j,i,1-k,1-l,1-m,o,n\}$, so that these two cases are equivalent.
    Then up to relabeling of variables, $F=\{0,0,0,0,0,0,0\}$.
\end{proof}

In the fourth lemma, we show what restrictions we must have on a lattice whose join-meet ideal contains the 2-minors of the ladder

\vspace{12pt}

\begin{center}
$\begin{matrix}
a & b & e & \\
c & d & f & i\\
 & g & h & j\\
 & & k & l
\end{matrix}$
\end{center}

 \noindent as standard generators, and which meets certain comparability conditions.

\vspace{12pt}

\begin{lemma}\label{9rel}
    Suppose
    \begin{eqnarray*}
    s_1&=&bc-ad\\
    s_2&=&ce-af\\
    s_3&=&de-bf\\
    s_4&=&eg-bh\\
    s_5&=&fg-dh\\
    s_6&=&gi-dj\\
    s_7&=&hi-fj\\
    s_8&=&ik-fl\\
    s_9&=&jk-hl
    \end{eqnarray*}
    are standard generators of a lattice $D$, and that $\{a,g\}$, $\{a,h\}$, $\{c,g\}$, $\{c,h\}$, $\{b,i\}$, $\{b,j\}$, $\{e,i\}$, $\{e,j\}$, $\{d,k\}$, $\{d,l\}$, $\{g,k\}$, and $\{g,l\}$ are comparable pairs in $D$. Let $F\in \F_2^{9}$ be defined for these nine elements as in Definition~\ref{F}.  Then $F=\{0,0,0,0,0,0,0,0,0\}$.
\end{lemma}

\begin{proof}
     We first note that with natural relabeling, both $\{s_1,s_2,s_3,s_4,s_5,s_6,s_7\}$ and \\$\{s_3,s_4,s_5,s_6,s_7,s_8,s_9\}$ satisfy the hypotheses of Lemma~\ref{7rel}, so if we let $F$ be defined as in Definition~\ref{F}, this limits the cases we need to consider to 9-tuples whose first seven entries and whose last seven entries satisfy the conclusion of Lemma~\ref{7rel}. We see by Lemma~\ref{7rel} that $F=\{0,0,0,0,0,0,0,0,0\}$.  
\end{proof}

We now have the machinery necessary to show that for $t=(1,1,1,1,1,0)$, $I_{G_{5}^t}$ does not come from a lattice.  In our proof, we use the previous four lemmas and the fact that $I_{G_{5}^t}$ is generated by the distinguished minors of $L_5^{(1,1,1,1,1,0)}$: 

\vspace{12pt}

\begin{center}
$\begin{matrix}
x_0 & x_2 & x_5 & & \\
x_3 & x_4 & x_6 & x_9 & x_{13}\\
 & x_7 & x_8 & x_{10} & \\
 & & x_{11} & x_{12} & x_{14}
\end{matrix}$
\end{center}

\vspace{12pt}

\begin{prop}\label{nolattice}
    Let $n=5$ and $t=(1,1,1,1,1,0)$.  Then the set of standard generators for $I_{G_{5}^t}$ is not equal to the complete set of standard generators (up to sign) of any (classical) lattice. 
\end{prop}

\begin{proof}
     By Remark~\ref{ti} and choice of $n=5$ and $t=(1,1,1,1,1,0)$, the generators of $I_{G_{5}^t}$ are 
     \begin{align*}
        s_1  &=  x_2x_3-x_0x_4 & s_2  &=  x_3x_5-x_0x_6\\
        s_3 &= x_4x_5-x_2x_6 &s_4 &= x_5x_7-x_2x_8\\
        s_5 &= x_6x_7-x_4x_8 &s_6 &= x_7x_9-x_4x_{10}\\
        s_7 &= x_8x_9-x_6x_{10}  &s_8 &= x_9x_{11}-x_6x_{12}\\
        s_9 &= x_{10}x_{11}-x_8x_{12} &s_{10} &= x_{11}x_{13}-x_6x_{14}\\
        s_{11} &= x_{12}x_{13}-x_9x_{14} & &
     \end{align*}
     
    Suppose a lattice $D$ exists whose complete set of standard generators (up to sign) equals \{$s_1,\ldots,s_{11}$\}.  We note that if the monomial $x_ix_j$ does not appear in any of the $s_k$, then $\{x_i,x_j\}$ is a comparable pair, since otherwise $\pm(x_ix_j-(x_i\vee x_j)(x_i\wedge x_j))$ would be in the set of standard generators of $D$.  Thus the pairs 
    $\{x_0,x_7\}$, $\{x_0,x_8\}$, $\{x_3,x_7\}$, $\{x_3,x_8\}$, $\{x_2,x_9\}$, $\{x_2,x_{10}\}$, $\{x_5,x_9\}$, $\{x_5,x_{10}\}$, $\{x_4,x_{11}\}$, $\{x_4,x_{12}\}$, $\{x_7,x_{11}\}$, $\{x_7,x_{12}\}$, $\{x_{10},x_{13}\}$, and $\{x_{10},x_{14}\}$ are comparable pairs in $D$. Let $F\in \F_2^{11}$ be defined as in Definition~\ref{F}.  Then with natural relabeling of the first nine relations, this lattice satisfies the hypotheses of Lemma~\ref{9rel}, so the only cases we need to consider are $F=\{0,0,0,0,0,0,0,0,0,a,b\}$.  
    
    Since $F_1=0$, without loss of generality, we have $x_0>\{x_2,x_3\}>x_4$.  Then with the ordering chosen, the fact that $F_3=F_5=F_7=F_9=0$ yields
     \begin{align*}
        &x_2>\{x_4,x_5\}>x_6\\
        &x_4>\{x_6,x_7\}>x_8\\
        &x_6>\{x_8,x_9\}>x_{10}\\
        &x_8>\{x_{10},x_{11}\}>x_{12}.
     \end{align*}
     The reader may verify that $b=0$ and $b=1$ both yield contradictions based on inspecting the standard generator $s_{11}$ in light of the comparability of the pairs $\{x_{10},x_{13}\}$ and $\{x_{10},x_{14}\}$, respectively, using the same technique employed in Case 1 of the proof of Lemma~\ref{5rel}.
     
     We conclude that there is no lattice whose complete set of standard generators (up to sign) equals the set of standard generators of $I_{G_{5}^{(1,1,1,1,1,0)}}$.
\end{proof}

\section{Properties of the Family of Toric Rings} 
\label{resultsprops}

In Section~\ref{family}, we defined a family of toric rings, the rings $R(n,t)$ coming from the family $\mathcal{F}$, and we demonstrated some context for these rings in the area of graph theory.  Now we investigate some of the algebraic properties of each $R(n,t)$.  We develop proofs to establish dimension, regularity, and multiplicity.
\subsection{Dimension and System of Parameters}

We use the the degree reverse lexicographic monomial order with $x_0>x_2>x_3>\cdots$ throughout this section, and denote it by $>$. We show that the standard generators $s_k$ given in Remark~\ref{ti} are a Gr\"obner basis for $I_{G_n^t}$ with respect to $>$.

\begin{lemma}\label{Grobner}
If $s_1,\ldots,s_{2n+1}$ are as in Remark~\ref{ti}, then $B=\{s_1,\ldots,s_{2n+1}\}$ is a Gr{\"o}bner basis for $I_{G_n^t}$  with respect to $>$.
\end{lemma}

\begin{proof}
This is a straightforward computation using Buchberger's Criterion and properties of the $j_k$ from Remark~\ref{jk}.  By Remark~\ref{ti}, for $1\leq i \leq n$ the ideal $I_{G_n^t}$ is generated by
    \begin{eqnarray*}
    s_{1}&=&x_{2}x_{3}-x_{j_{1}}x_{4}\\
    s_{2i}&=&x_{2i+1}x_{2i+3}-x_{j_{2i}}x_{2i+4}\\
    s_{2i+1}&=&x_{2i+2}x_{2i+3}-x_{j_{2i+1}}x_{2i+4},
    \end{eqnarray*}
    where $j_1=j_2=0$, $j_3=2$, and for $i\geq 2$, we have
    \begin{eqnarray*}
    t_{i+1}&=&0\iff j_{2i}=j_{2i-2}\iff j_{2i+1}=2i-1\\
    t_{i+1}&=&1\iff j_{2i}=j_{2i-1}\iff j_{2i+1}=2i.
    \end{eqnarray*}

If we adopt $S$-polynomial notation $S_{i,j}$ for the $S$-polynomial of $s_i$ and $s_j$, then the cases to consider are 
\[
S_{2i-1,2i}\text{ and } S_{2i,2i+1}\text{ for }1\leq i\leq n
\]
\[
S_{2i,2i+2}\text{ for }1\leq i\leq n-1.
\]
To give a flavor of the computation involved, we show the case $S_{2i-1,2i}$ for $1\leq i\leq n$, and leave the remaining cases to the reader.  We show in each subcase that $S_{2i-1,2i}$ is equal to a sum of basis elements with coefficients in $S(n)$, so that the reduced form of $S_{2i-1,2i}$ is zero in each subcase.  We have 
\begin{eqnarray*}
S_{2i-1,2i}&=&x_{2i+3}(x_{2i}x_{2i+1}-x_{j_{2i-1}}x_{2i+2})-x_{2i}(x_{2i+1}x_{2i+3}-x_{j_{2i}}x_{2i+4})\\
&=&-x_{j_{2i-1}}x_{2i+2}x_{2i+3}+x_{j_{2i}}x_{2i}x_{2i+4}
\end{eqnarray*}

\vspace{12pt}

Case 1: If $i\geq 2$ and $t_{i+1}=0$, then $j_{2i}=j_{2i-2}$ and $j_{2i+1}=2i-1$, so we have
\begin{eqnarray*}
S_{2i-1,2i}&=&-x_{j_{2i-1}}x_{2i+2}x_{2i+3}+x_{j_{2i-2}}x_{2i}x_{2i+4}
\end{eqnarray*}

\begin{quote}
    Case 1.1: If in addition $i\geq 3$ and $t_{i}=0$, then $j_{2i-2}=j_{2i-4}$ and $j_{2i-1}=2i-3$, so we have
    \begin{eqnarray*}
    S_{2i-1,2i}&=&-x_{2i-3}x_{2i+2}x_{2i+3}+x_{j_{2i-4}}x_{2i}x_{2i+4}\\
    &=&-x_{2i-3}(x_{2i+2}x_{2i+3}-x_{j_{2i+1}}x_{2i+4})+x_{2i+4}(x_{2i-3}x_{2i-1}-x_{j_{2i-4}}x_{2i})\\
    &=&-x_{2i-3}s_{2i+1}+x_{2i+4}s_{2i-4}.
    \end{eqnarray*}
\end{quote}

\begin{quote}
    Case 1.2: If in addition $i=2$ or $i\geq 3$ and $t_{i}=1$, then $j_{2i-2}=j_{2i-3}$ and $j_{2i-1}=2i-2$, so we have
    \begin{eqnarray*}
    S_{2i-1,2i}&=&-x_{2i-2}x_{2i+2}x_{2i+3}+x_{j_{2i-3}}x_{2i}x_{2i+4}\\
    &=&-x_{2i-2}(x_{2i+2}x_{2i+3}-x_{j_{2i+1}}x_{2i+4})+x_{2i+4}(x_{2i-2}x_{2i-1}-x_{j_{2i-3}}x_{2i})\\
    &=&-x_{2i-2}s_{2i+1}+x_{2i+4}s_{2i-3}.
    \end{eqnarray*}
\end{quote}

Case 2: If $i=1$ or if $i\geq 2$ and $t_{i+1}=1$, then $j_{2i}=j_{2i-1}$ and $j_{2i+1}=2i$, so we have
\begin{eqnarray*}
S_{2i-1,2i}&=&-x_{j_{2i}}x_{2i+2}x_{2i+3}+x_{j_{2i}}x_{2i}x_{2i+4}\\
&=&-x_{j_{2i}}s_{2i+1}.
\end{eqnarray*}

This concludes the case $S_{2i-1,2i}$ for $1\leq i\leq n-1$.  The remaining cases are similar.

\end{proof}

\begin{cor}\label{koszul}
The ring $R(n,t)$ is Koszul for all $n$ and all $t$.
\end{cor}

\begin{proof}
    Since $I_{G_n^t}$ has a quadratic Gr\"obner basis, the ring $R(n,t)$ is Koszul for all $n$ and all $t$ due to \cite[Th 2.28]{herzog-hibi-ohsugi}.
\end{proof}

\begin{cor}\label{initial}
    The initial ideal for $I_{G_n^t}$ with respect to the degree reverse lexicographic monomial order $>$ is 
    \begin{equation*}
        \init I_{G_n^t}=(x_{2}x_{3},\{x_{2i+1}x_{2i+3},x_{2i+2}x_{2i+3} \mid 1\leq i\leq n\}).
    \end{equation*}
\end{cor}

\noindent We note that $\init I_{G_n^t}$ does not depend on $t$, which will be useful for the following sections. 

We use the initial ideal $\init I_{G_n^t}$ from Corollary~\ref{initial} and direct computation to show the Krull dimension of $R(n,t)$. As a corollary, we obtain the projective dimension of $R(n,t)$.  We note that the Krull dimension, like the initial ideal, does not depend on $t$.  We refer the reader to Remark~\ref{ti} for a reminder of how to think of the toric ring \[
R(n,t)=\frac{S(n)}{I_{G_n^t}}=\frac{k[x_0,x_2,x_3,\ldots, x_{2n+4}]}{I_{G_n^t}}
\]
in the context of this work.

\begin{thm}\label{dimension}

The Krull dimension of $R(n,t)$ is 
\[
\dim R(n,t)= n+3.
\]

\end{thm}

\begin{proof}
Let $>$ be the degree reverse lexicographic monomial order with
\[
x_0>x_2>x_3>\cdots>x_{2n+4}.
\]
By Corollary~\ref{initial}, the initial ideal of $I_{G_n^t}$ with respect to $>$ is 
\begin{equation*}
        \init I_{G_n^t}=(x_{2}x_{3},\{x_{2i+1}x_{2i+3},x_{2i+2}x_{2i+3} \mid 1\leq i\leq n\}).
    \end{equation*}

Since $S(n)/(\init I_{G_n^t})$ and  $R(n,t)=S(n)/I_{G_n^t}$ are known to have the same Krull dimension (see for example \cite[Props 9.3.4 and 9.3.12]{cox-little-oshea}),
it suffices to prove that 
\[
\dim S(n)/(\init I_{G_n^t}) = n+3.
\]

To see that the dimension is at least $n+3$, we construct a chain of prime ideals in $S(n)$ containing $\init I_{G_n^t}$.  Since every monomial generator of $\init I_{G_n^t}$ contains a variable of odd index, we begin with  $P_{n}=(\{x_{k} \mid \text{$k$ odd, $2<k<2n+4$}\})$, a prime ideal containing $\init I_{G_n^t}$. Then we have the chain of prime ideals  $P_{n}\subsetneq P_{n}+(x_{0})\subsetneq P_{n}+(x_{0},x_{2})\subsetneq P_{n}+(x_{0},x_{2},x_{4})\subsetneq\cdots\subsetneq P_{n}+(\{x_{2i} \mid 0\leq i \leq n+2\})$, so that
\[
\dim S(n)/(\init I_{G_n^t})\geq n+3. 
\]

To see that the dimension is at most $n+3$, we find a sequence of $n+3$ elements in $S(n)/(\init I_{G_n^t})$ such that the quotient by the ideal they generate has dimension zero. Let
\[
X_n=x_0,x_2-x_3,x_4-x_5,\ldots,x_{2n}-x_{2n+1},x_{2n+2}-x_{2n+3},x_{2n+4}
\]
in $S(n)$, and take the quotient of $S(n)/(\init I_{G_n^t})$ by the image of $X_n$ to obtain the following.  In the last step, we rewrite the quotient of $S(n)$ and $(\init I_{G_n^t})+(X_n)$ by $(X_n)$ by setting $x_0$ and $x_{2n+4}$ equal to zero and replacing $x_{2i}$ with $x_{2i+1}$ for $1\leq i\leq n+1$:
\begin{eqnarray*}
    \frac{S(n)/(\init I_{G_n^t})}{((\init I_{G_n^t})+(X_n))/(\init I_{G_n^t})}&\cong&\frac{S(n)}{((\init I_{G_n^t})+(X_n))}\\
    &\cong&\frac{S(n)/(X_n)}{((\init I_{G_n^t})+(X_n))/(X_n)}\\
    &\cong&\frac{k[x_3,x_5,\ldots,x_{2n+1},x_{2n+3}]}{(x_3^2,\{x_{2i+1}x_{2i+3},x_{2i+3}^2\mid 1\leq i\leq n\})}.
\end{eqnarray*}
Since
\[
\sqrt{(x_3^2,\{x_{2i+1}x_{2i+3},x_{2i+3}^2\mid 1\leq i\leq n\})}=(x_3,x_5,\ldots,x_{2n+3}),
\]
the above ring has dimension zero.
Thus, 
\[
\dim S(n)/(\init I_{G_n^t}) \leq n+3.
\]

We conclude that $\dim R(n,t) = \dim S(n)/(\init I_{G_n^t})=n+3$.
\end{proof}

\begin{cor}\label{pdimension}

The projective dimension of $R(n,t)$ over $S(n)$ is 
\[
\pdim_{S(n)} R(n,t)= n+1.
\]

\end{cor}

\begin{proof}
    We know the Krull dimension of the polynomial ring $S(n)$ is $2n+4$.  The result follows from the fact that $R(n,t)$ is Cohen-Macaulay (Corollary~\ref{CM}) and from the graded version of the Auslander-Buchsbaum formula.
\end{proof}

\begin{rmk}\label{xtsop}
    The proof of the previous theorem shows that the image of
    \[
    X_n=x_{0},x_{2}-x_{3},x_{4}-x_{5},\ldots,x_{2n}-x_{2n+1},x_{2n+2}-x_{2n+3},x_{2n+4}
    \]
    in $S(n)/(\init I_{G_n^t})$ is a system of parameters for $S(n)/(\init I_{G_n^t})$.  We prove in the next theorem that the image of $X_n$ in $R(n,t)$ (which we call $\overline{X_n}$) is also a system of parameters for $R(n,t)$.  Before doing so, we introduce some notation and a definition which allows us to better grapple with the quotient ring $R(n,t)/(\overline{X_n})$.
\end{rmk}

\begin{defn}\label{widehatnotation}
    Consider the isomorphism
    \[
    \frac{R(n,t)}{(\overline{X_n})}=\frac{S(n)/(I_{G_n^t})}{((I_{G_n^t})+(X_n))/(I_{G_n^t})}\cong\frac{S(n)/(X_n)}{((I_{G_n^t})+(X_n))/(X_n)}
    \]
    We view taking the quotient by $X_n$ as setting $x_0$ and $x_{2n+4}$ equal to zero and replacing $x_{2i}$ with $x_{2i+1}$ for $1\leq i\leq n+1$ to obtain
    \begin{equation*}
    \widehat{S(n)}:=k[x_3,x_5,\ldots,x_{2n+1},x_{2n+3}]\cong S(n)/(X_n).
    \end{equation*}
    By the same process (detailed below), we obtain the ideal $\widehat{I_{G_n^t}}\cong(I_{G_n^t}+(X_n))/(X_n)$.
    We further define the quotient
    \begin{equation*}
    \widehat{R(n,t)}:=\widehat{S(n)}/\widehat{I_{G_n^t}}\cong R(n,t)/(\overline{X_n}).
    \end{equation*}
    We find this notation natural since it is often used for the removal of variables, and the quotient by $X_n$ may be viewed as identifying and removing variables.  Since this work has no completions in it, there should be no conflict of notation.
\end{defn}

\begin{defn}\label{widehati}

To define $\widehat{I_{G_n^t}}$ in particular, we recall the standard generators of $I_{G_n^t}$ and introduce further notation to describe the generators of $\widehat{I_{G_n^t}}\cong(I_{G_n^t}+(X_n))/(X_n)$.  By Remark~\ref{ti}, the standard generators of $I_{G_n^t}$ are 
    \begin{eqnarray*}
    s_1&=&x_{2}x_{3}-x_{j_{1}}x_{4}\\
    s_{2i}&=&x_{2i+1}x_{2i+3}-x_{j_{2i}}x_{2i+4}\\
    s_{2i+1}&=&x_{2i+2}x_{2i+3}-x_{j_{2i+1}}x_{2i+4},
    \end{eqnarray*}
    for $1\leq i\leq n$, where the nonnegative integers $j_k$ are as in Remark~\ref{jk}.

Let $\widehat{\iota}$ be the largest index such that $j_{2\widehat{\iota}}=0$.  By Remark~\ref{jk}, we see that the $j_{2i}$ are defined recursively and form a non-decreasing sequence.  Then
\begin{equation*}
    j_1=j_2=j_4=j_6=\cdots=j_{2\widehat{\iota}}=0,
\end{equation*}
and since we view taking the quotient by $X_n$ as setting $x_0$ and $x_{2n+4}$ equal to zero and replacing $x_{2i}$ with $x_{2i+1}$ for $1\leq i\leq n+1$, we define $\widehat{I_{G_n^t}}$ by replacing $x_{j_k}$ with $x_{J_k}$ (defined below) for $1\leq k<2n$ to obtain
\begin{eqnarray*}
    \widehat{s_1}&=&x_{3}^2-x_{J_{1}}x_{5}\\
    \widehat{s_{2i}}&=&x_{2i+1}x_{2i+3}-x_{J_{2i}}x_{2i+5}\\
    \widehat{s_{2i+1}}&=&x_{2i+3}^2-x_{J_{2i+1}}x_{2i+5}\\
    \widehat{s_{2n}}&=&x_{2n+1}x_{2n+3}\\
    \widehat{s_{2n+1}}&=&x_{2n+3}^2
\end{eqnarray*}
for $1\leq i< n$, where
\[
x_{J_k}=\begin{cases}
            0&\text{ if }k\text{ is even and }k\leq2\widehat{\iota}, \text{ or if }k=1\\
            x_{j_k+1} &\text{ if } 2\widehat{\iota}<k<2n \text{ and $j_k$ is even}\\
            x_{j_k} &\text{ if }j_k\text{ is odd}
        \end{cases} 
\]
We note that $J_k\leq k$ for each $1\leq k< 2n$, since $j_k\leq k-1$ by Remark~\ref{jk}. 
By properties of the original $j_k$ from Remark~\ref{jk}, we know that $x_{J_1}=x_{J_2}=0$, $J_3=3$, and for $2\leq i<n$,
    \begin{eqnarray*}
        t_{i+1}&=&0\iff x_{J_{2i}}=x_{J_{2i-2}}\iff J_{2i+1}=2i-1\\
        t_{i+1}&=&1\iff x_{J_{2i}}=x_{J_{2i-1}}\iff J_{2i+1}=2i+1.
    \end{eqnarray*}
\end{defn}

\begin{exmp}
    We construct the ring $R(2,(0,0,0))$. For the graph $G_2^{(0,0,0)}\in \mathcal{F}$, we have the toric ring
    \[
    R(2,(0,0,0))=\frac{k[x_0,x_2,x_3,\ldots,x_{8}]}{(x_2x_3-x_0x_4,x_3x_5-x_0x_6,x_4x_5-x_2x_6,x_5x_7-x_0x_8,x_6x_7-x_3x_8)}
    \]
    coming from the ladder-like structure
\[
L_2^{(0,0,0)}=
\begin{matrix}
x_0 & x_2 & x_5\\
x_3 & x_4 & x_6\\
x_7 &     & x_8
\end{matrix}
\]
from Example~\ref{mexmp1}.  We know 
\[
X_2=x_{0},x_{2}-x_{3},x_{4}-x_{5},\ldots,x_4-x_{5},x_{6}-x_{7},x_{8},
\]
so that $R(2,(0,0,0))/(\overline{X_2})$ is isomorphic to 
\begin{equation*}
    \frac{k[x_0,x_2,x_3,x_4,x_5,x_6,x_7,x_{8}]}{(x_2x_3-x_0x_4,x_3x_5-x_0x_6,x_4x_5-x_2x_6,x_5x_7-x_0x_8,x_6x_7-x_3x_8,x_{0},x_{2}-x_{3},\ldots,x_{8})}
\end{equation*}
\begin{equation*}
    \cong\frac{k[x_3,x_5,x_7]}{(x_3^2,x_3x_5,x_5^2-x_3x_7,x_5x_7,x_7^2)}=\widehat{R(2,(0,0,0))}.
\end{equation*}
\end{exmp}

Now we show that $X_n$ is also a system of parameters for $R(n,t)$, and not just for the quotient by the initial ideal.

\begin{prop}\label{sopedgering}

Let $R(n,t)=S(n)/I_{G_n^t}$ and let
\begin{equation*}
    X_n=x_{0},x_{2}-x_{3},x_{4}-x_{5},\ldots,x_{2n}-x_{2n+1},x_{2n+2}-x_{2n+3},x_{2n+4}
\end{equation*}
so that the image of $X_n$ in $S(n)/(\init I_{G_n^t})$ is the system of parameters from Remark~\ref{xtsop}. Then the image of $X_n$ in $R(n,t)$ is a system of parameters for $R(n,t)$.  

\end{prop}

\begin{proof}

Let $X_n$ be defined as above.  Then by Theorem~\ref{dimension} and Definition~\ref{widehatnotation} we need only show that $\dim \widehat{R(n,t)}=0$.  We have for $n=0$ that 

\begin{equation*}
    \widehat{R(0,t)}=\frac{k[x_3]}{(x_3^2)},
\end{equation*}
for $n=1$
\begin{equation*}
    \widehat{R(1,t)}=\frac{k[x_3,x_5]}{(x_3^2,x_3x_5,x_5^2)},
\end{equation*}
and for $ n>1$
\begin{equation*}
    \widehat{R(n,t)}=\frac{\widehat{S(n)}}{\widehat{I_{G_n^t}}}=\frac{k[x_3,x_5,\ldots,x_{2n+1},x_{2n+3}]}{(\{\widehat{s_1},\widehat{s_{2i}},\widehat{s_{2i+1}}\mid 1\leq i\leq n\})},
\end{equation*}
where
\begin{eqnarray*}
    \widehat{s_1}&=&x_{3}^2\\
    \widehat{s_{2i}}&=&x_{2i+1}x_{2i+3}-x_{J_{2i}}x_{2i+5}\\
    \widehat{s_{2i+1}}&=&x_{2i+3}^2-x_{J_{2i+1}}x_{2i+5}\\
    \widehat{s_{2n}}&=&x_{2n+1}x_{2n+3}\\
    \widehat{s_{2n+1}}&=&x_{2n+3}^2
\end{eqnarray*}
for $1\leq i < n$ from Definition~\ref{widehati}.
We know $\dim \widehat{R(n,t)}=\dim \widehat{S(n)}/\widehat{I_{G_n^t}}=\dim \widehat{S(n)}/\sqrt{\widehat{I_{G_n^t}}}$. We claim that
\begin{equation*}
    \sqrt{\widehat{I_{G_n^t}}}=\left(x_{3},x_{5},\ldots,x_{2n+1},x_{2n+3}\right).
\end{equation*}
This is clear for $n\in\{0,1\}$.  For $n >1$, we prove this by induction.  Since $\widehat{s_1}=x_3^2$ and $\widehat{s_{2n+1}}=x_{2n+3}^2$ are in $\widehat{I_{G_n^t}}$, we have $x_{3}, x_{2n+3}\in\sqrt{\widehat{I_{G_n^t}}}$. Since
\[
\widehat{s_{3}}=x_5^2-x_3x_7\in \widehat{I_{G_n^t}}\subseteq \sqrt{\widehat{I_{G_n^t}}}
\]
and $x_{3}\in\sqrt{\widehat{I_{G_n^t}}}$, we get $x_{5}^2\in\sqrt{\widehat{I_{G_n^t}}}$, so that $x_{5}\in\sqrt{\widehat{I_{G_n^t}}}$. Now suppose $x_{2i-1},x_{2i+1}\in\sqrt{\widehat{I_{G_n^t}}}$ for $2\leq i< n$. We have
\[
\widehat{s_{2i+1}}=x_{2i+3}^{2}-x_{J_{2i+1}}x_{2i+5}\in\widehat{I_{G_n^t}}\subseteq\sqrt{\widehat{I_{G_n^t}}}.
\]
But $x_{J_{2i+1}}\in\{x_{2i-1},x_{2i+1}\}$ by Definition~\ref{widehati} and $\{x_{2i-1},x_{2i+1}\}\subseteq\sqrt{\widehat{I_{G_n^t}}}$ by induction, so that $x_{2i+3}^{2}\in\sqrt{\widehat{I_{G_n^t}}}$, and hence $x_{2i+3}\in\sqrt{\widehat{I_{G_n^t}}}$. We conclude that 
\[
\left(x_{3},x_{5},\ldots,x_{2n+1},x_{2n+3}\right)\subseteq\sqrt{\widehat{I_{G_n^t}}}\subseteq\left(x_{3},x_{5},\ldots,x_{2n+1},x_{2n+3}\right),
\]
so we have equality. Since $\widehat{S(n)}/\sqrt{\widehat{I_{G_n^t}}}\cong k$ has dimension zero, so does $\widehat{R(n,t)}\cong R(n,t)/(\overline{X_n})$.  Thus, the image of $X_n$ is a system of parameters for $R(n,t)$.
\end{proof}

\begin{rmk}\label{art}
We note that as a consequence of the proof of the preceding theorem, the ring $\widehat{R(n,t)}$ is Artinian, which will be relevant in Section~\ref{lmr}.
\end{rmk}

\begin{cor}\label{xt}
The image of 
\begin{equation*}
    X_n=x_{0},x_{2}-x_{3},x_{4}-x_{5},\ldots,x_{2n}-x_{2n+1},x_{2n+2}-x_{2n+3},x_{2n+4}
\end{equation*}
in $R(n,t)=S(n)/I_{G_n^t}$ is a regular sequence for $R(n,t)$.  
\end{cor}

\begin{proof}

We know by Proposition~\ref{sopedgering} that the image of $X_n$ in $R(n,t)$ is a linear system of parameters.  Since the rings $R(n,t)$ are Cohen-Macaulay (Corollary~\ref{CM}), we are done.
\end{proof}

\subsection{Length, Multiplicity, and Regularity} \label{lmr}

In this section, we determine the multiplicity and Castelnuovo-Mumford regularity of the toric rings $R(n,t)$ coming from the associated graphs $G_n^t\in\mathcal{F}$ by computing the length of the Artinian rings
\[
\widehat{R(n,t)}\cong R(n,t)/(\overline{X_n})
\]
from Definition~\ref{widehatnotation}. We know by Corollary~\ref{xt} that $\overline{X_n}$ is a linear regular sequence for $R(n,t)$, which allows us to compute the multiplicity of the rings $R(n,t)$. As a corollary of Theorem~\ref{Hilbert Series mod reg seq}, which establishes the Hilbert function for $\widehat{R(n,t)}$, we obtain the multiplicity and regularity of $R(n,t)$.  We also develop an alternate graph-theoretic proof for the regularity of $R(n,t)$, which is included at the end of this section.  

We begin with a lemma establishing a vector space basis for $\widehat{R(n,t)}$, which we use extensively for our results.

\begin{lemma}\label{uniquerep}

The image of all squarefree monomials with only odd indices whose indices are at least four apart, together with the image of $1_k$, forms a vector space basis for $\widehat{R(n,t)}$  over $k$.

\end{lemma}

\begin{proof}

We recall for the reader the definition of $\widehat{R(n,t)}$ and then find the initial ideal of $\widehat{I_{G_n^t}}$ and use Macaulay's Basis Theorem to show that the desired representatives form a basis for $\widehat{R(n,t)}$ as a vector space over $k$.

From Definition~\ref{widehatnotation}, we have
\begin{equation*}
    R(n,t)/(\overline{X_n})\cong\widehat{R(n,t)}= \widehat{S(n)}/\widehat{I_{G_n^t}},
\end{equation*}
where
\begin{equation*}
    \widehat{S(n)}=k[x_3,x_5,\ldots,x_{2n+1},x_{2n+3}].
\end{equation*}
By Definition~\ref{widehati}, for $1\leq i< n$ the ideal $\widehat{I_{G_n^t}}$ is generated by 
\begin{eqnarray*}
    \widehat{s_1}&=&x_{3}^2-x_{J_{1}}x_{5}\\
    \widehat{s_{2i}}&=&x_{2i+1}x_{2i+3}-x_{J_{2i}}x_{2i+5}\\
    \widehat{s_{2i+1}}&=&x_{2i+3}^2-x_{J_{2i+1}}x_{2i+5}\\
    \widehat{s_{2n}}&=&x_{2n+1}x_{2n+3}\\
    \widehat{s_{2n+1}}&=&x_{2n+3}^2,
\end{eqnarray*}
where $x_{J_1}=x_{J_2}=0$, $J_3=3$, and for $2\leq i<n$, 
    \begin{eqnarray*}
        t_{i+1}&=&0\iff x_{J_{2i}}=x_{J_{2i-2}}\iff J_{2i+1}=2i-1\\
       t_{i+1}&=&1\iff x_{J_{2i}}=x_{J_{2i-1}}\iff J_{2i+1}=2i+1.
    \end{eqnarray*}

We first show that the image of the monomials with the desired property is a basis in the quotient of $\widehat{S(n)}$ by the initial ideal $\init \widehat{I_{G_n^t}}$.  By Macaulay's Basis Theorem, which is Theorem 1.5.7 in \cite{cca1-kreuzer-robbiano}, the image of these monomials in $\widehat{R(n,t)}$ is also a basis. 

To find the initial ideal of $\widehat{I_{G_n^t}}$, we establish that the given generators $\widehat{s_k}$ are a  Gr\"{o}bner basis for $\widehat{I_{G_n^t}}$ with respect to the degree reverse lexicographic order $>$.  This is a relatively straightforward computation using Buchberger's Criterion, with separate cases for when $x_{J_k}=0$.

If we adopt $S$-polynomial notation $S_{i,j}$ for the $S$-polynomial of $\widehat{s_i}$ and $\widehat{s_j}$, then the cases to consider are 
\[
S_{1,2},S_{2,3},S_{2,4},S_{2n-2,2n},S_{2n-1,2n},S_{2n,2n+1}
\]
\[
S_{2i-1,2i}\text{ for }1<i<n
\]
\[
S_{2i,2i+1}\text{ for }1<i<n
\]
\[
S_{2i,2i+2}\text{ for }1<i<n-1.
\]
To give a flavor of the computation involved, we show the case $S_{2i,2i+2}$ for $1<i<n-1$, and leave the remaining cases to the reader.  We show in each subcase that $S_{2i,2i+2}$ is equal to zero or to a sum of basis elements with coefficients in $\widehat{S(n)}$, so that the reduced form of $S_{2i,2i+2}$ is zero in each subcase.  We have
\begin{eqnarray*}
S_{2i,2i+2}&=&x_{2i+5}(x_{2i+1}x_{2i+3}-x_{J_{2i}}x_{2i+5})-x_{2i+1}(x_{2i+3}x_{2i+5}-x_{J_{2i+2}}x_{2i+7})\\
&=&-x_{J_{2i}}x_{2i+5}^2+x_{2i+1}x_{J_{2i+2}}x_{2i+7}
\end{eqnarray*}

\vspace{12pt}

Case 1: If $t_{i+2}=0$, then $x_{J_{2i+2}}=x_{J_{2i}}$ and $J_{2i+3}=2i+1$, so we have
\begin{eqnarray*}
S_{2i,2i+2}&=&-x_{J_{2i}}x_{2i+5}^2+x_{J_{2i}}x_{2i+1}x_{2i+7}\\
&=&-x_{J_{2i}}\widehat{s_{2i+3}}
\end{eqnarray*}

We note that if $x_{J_{2i}}=x_{J_{2i+2}}=0$, then $S_{2i,2i+2}=0$.

\vspace{12pt}

Case 2: If $t_{i+2}=1$, then $x_{J_{2i+2}}=x_{J_{2i+1}}$ and $J_{2i+3}=2i+3$, so we have
\begin{eqnarray*}
S_{2i,2i+2}&=&-x_{J_{2i}}x_{2i+5}^2+x_{2i+1}x_{J_{2i+1}}x_{2i+7}.
\end{eqnarray*}

\begin{quote}
Case 2.1: If in addition $t_{i+1}=0$, then $x_{J_{2i}}=x_{J_{2i-2}}$ and $J_{2i+1}=2i-1$, so we have
\begin{eqnarray*}
S_{2i,2i+2}&=&-x_{J_{2i-2}}x_{2i+5}^2+x_{2i+1}x_{2i-1}x_{2i+7}\\
&=&-x_{J_{2i-2}}(x_{2i+5}^2-x_{J_{2i+3}}x_{2i+7})+x_{2i+7}(x_{2i-1}x_{2i+1}-x_{J_{2i-2}}x_{2i+3})\\
&=&-x_{J_{2i-2}}\widehat{s_{2i+3}}+x_{2i+7}\widehat{s_{2i-2}}.
\end{eqnarray*}

We note that if $x_{J_{2i}}=x_{J_{2i-2}}=0$, then $S_{2i,2i+2}=x_{2i+7}\widehat{s_{2i-2}}$.

\vspace{12pt}

\noindent Case 2.2: If in addition $t_{i+1}=1$, then $x_{J_{2i}}=x_{J_{2i-1}}$ and $J_{2i+1}=2i+1$, so we have
\begin{eqnarray*}
S_{2i,2i+2}&=&-x_{J_{2i-1}}x_{2i+5}^2+x_{2i+1}^2x_{2i+7}\\
&=&-x_{J_{2i-1}}(x_{2i+5}^2-x_{J_{2i+3}}x_{2i+7})+x_{2i+7}(x_{2i+1}^2-x_{J_{2i-1}}x_{2i+3})\\
&=&-x_{J_{2i-1}}\widehat{s_{2i+3}}+x_{2i+7}\widehat{s_{2i-1}}.
\end{eqnarray*}
\end{quote}

This concludes the case $S_{2i,2i+2}$ for $1<i<n-1$.  The remaining cases are similar.

Then the given generators $\widehat{s_k}$ are a Gr\"{o}bner Basis for $\widehat{I_{G_n^t}}$, so that the initial ideal is
\[
\init(\widehat{I_{G_n^t}})=(x_3^2,\{x_{2i+1}x_{2i+3},x_{2i+3}^2 \mid 1\leq i \leq n\} )
\]
in the ring $\widehat{S(n)}=k[x_3,x_5,\ldots,x_{2n+1},x_{2n+3}]$. Since $\init(\widehat{I_{G_n^t}})$ consists precisely of all squares of variables in $\widehat{S(n)}$ and all degree two products of variables whose indices differ by exactly two, it follows that the image of the squarefree monomials whose indices are at least four apart, together with the image of $1_k$, forms a basis for $\frac{\widehat{S(n)}}{\init\widehat{I_{G_n^t}}}$.  By Macaulay's Basis Theorem, the image of these monomials in  $\widehat{R(n,t)}=\frac{\widehat{S(n)}}{\widehat{I_{G_n^t}}}$ is also a basis. 

\end{proof}

We use the lemma above to establish facts about the vector space dimensions of degree $i$ pieces of $\widehat{R(n,t)}$, which are applied further below to establish length and multiplicity.

\begin{notation}\label{d}
    Throughout this section, we use $d_{n,i}:=\dim_k (\widehat{R(n,t)})_{i}$ for the vector space dimension of the degree $i$ piece of $\widehat{R(n,t)}$, that is, for the $i$th coefficient in the Hilbert series of $\widehat{R(n,t)}$.  By Lemma~\ref{uniquerep}, these are independent of $t$.
\end{notation}

We establish a recursive relationship between these dimensions by introducing a short exact sequence of vector spaces.

\begin{lemma}\label{recursion}
    For $n\geq 2$ and $i\geq 1$, the vector space dimension $d_{n,i}=\dim_k (\widehat{R(n,t)})_{i}$ satisfies the recursive relationship 
    \[
    d_{n,i}=d_{n-1,i}+d_{n-2,i-1}.
    \]
\end{lemma}

\begin{proof}
We use the vector space basis defined in Lemma~\ref{uniquerep}.  We note that the basis elements described are actually monomial representatives (which do not depend on $t$) of equivalence classes (which do depend on $t$), but we suppress this and speak as if they are monomials, not depending on $t$. We then take the liberty of suppressing $t$ in what follows, for convenience.  We recall for the reader that 
\begin{eqnarray*}
    \widehat{S(n)}&=&k[x_3,x_5,\ldots,x_{2n-3},x_{2n-1},x_{2n+1},x_{2n+3}]\\
    \widehat{S(n-1)}&=&k[x_3,x_5,\ldots,x_{2n-3},x_{2n-1},x_{2n+1}]\\
    \widehat{S(n-2)}&=&k[x_3,x_5,\ldots,x_{2n-3},x_{2n-1}]
\end{eqnarray*}

Let $x_{2n+3}:(\widehat{R(n-2)})_{i-1}\to (\widehat{R(n)})_{i}$ be multiplication by $x_{2n+3}$, and let\\ $\widehat{x_{2n+3}}:(\widehat{R(n)})_{i}\to (\widehat{R(n-1)})_{i}$ be defined for a basis element $b$ by 
\[
\widehat{x_{2n+3}}(b)=\begin{cases}
                            b &\text{if }x_{2n+3}\nmid b\\
                            0 &\text{if }x_{2n+3}\mid b.
                        \end{cases}
\]
We note that these vector space maps are well-defined, since $1_k$ or a squarefree monomial with odd indices at least four apart has an output of 0, $1$, or a monomial with the same properties.  The following sequence of vector spaces is exact

\[
\xymatrix{0\ar[r] & (\widehat{R(n-2)})_{i-1}\ar[r]^{x_{2n+3}} & (\widehat{R(n)})_{i}\ar[r]^{\widehat{x_{2n+3}}} & (\widehat{R(n-1)})_{i}\ar[r] & 0,}
\]
so that
\[
d_{n,i}=d_{n-1,i}+d_{n-2,i-1}.
\]
\end{proof}

Applying Lemma~\ref{recursion} and induction, we achieve the following closed formula for the coefficients of the Hilbert series of $\widehat{R(n,t)}$.  

\begin{thm}\label{Hilbert Series mod reg seq}

If $R(n,t)=S(n)/I_{G_n^t}$ and $\widehat{R(n,t)}\cong R(n,t)/(\overline{X_n})$, we have 
\begin{equation*}
    {\displaystyle \dim_k(\widehat{R(n,t)})_{i}=\begin{cases}
1 & i=0\\
{\displaystyle \frac{1}{i!}\prod_{j=1}^{i}(n+j-2(i-1))} & i\geq 1.
\end{cases}}
\end{equation*}
In particular, $\dim_k(\widehat{R(n,t)})_{i}=0$ when $i>n/2+1$.

\end{thm}

\vspace{0.25cm}

\begin{proof}
We begin with the proof of the last statement, and note that throughout the following proof we use Notation~\ref{d}. When $n\geq 0$ is even, $i=n/2+2$, and $j=2(\leq i)$, we have a factor of zero, and when $n\geq 1$ is odd, $i=n/2+3/2$, and $j=1(<i)$, we have a factor of zero. Thus $d_{n,i}=0$ when $i>n/2+1$.

Now we establish the base cases $i,n\in\{0,1\}$, then proceed by induction. It is clear that $d_{n,0}=1$, generated by $1_k$. By Lemma~\ref{uniquerep} and by the fact that $\widehat{R(n,t)}$ is a graded quotient, every nonzero element of positive degree $i$ can be represented uniquely as a sum of degree $i$ squarefree monomials with odd indices whose indices are at least four apart. Then $(\widehat{R(n,t)})_{1}$ is generated by the images of all the odd variables
\[
x_3,x_{2(1)+3},\ldots,x_{2n+3}
\]
in $S(n)$, so that  
\[
d_{n,1}=n+1=\frac{1}{1!}\prod_{j=1}^{1}(n+j-2(1-1))
\]
matches the given formula.

Now we establish the base cases $n=0$ and $n=1$ for all $i$.
We recognize that the first monomial of degree two with odd indices at least four apart is $x_3x_7$, which does not exist until $n=2$, so we have

\begin{equation*}
d_{0,i}=\begin{cases}
1 & i=0\\
1 & i=1\\
0 & \text{else}
\end{cases}
=
\begin{cases}
1 & i= 0\\
{\displaystyle \frac{1}{i!}\prod_{j=1}^{i}(j-2(i-1))} & i \geq 1
\end{cases}
\end{equation*}
and 
\begin{equation*}
d_{1,i}=\begin{cases}
1 & i=0\\
2 & i=1\\
0 & \text{else}
\end{cases}
=
\begin{cases}
1 & i=0\\
{\displaystyle \frac{1}{i!}\prod_{j=1}^{i}(1+j-2(i-1))} & i \geq 1,
\end{cases}
\end{equation*}
which match the given formula.

This gives us the following table of base cases for $d_{n,i}$, which match the given formula:
\begin{center}
\begin{tabular}{|c|c|c|c|c|c|c|c|c|c|}
    \hline
    $n\setminus i$&0&1&2&3&4&5&6&7&$\cdots$\\
    \hline
    0&1&1&0&0&0&0&0&0&$\cdots$\\
    \hline
    1&1&2&0&0&0&0&0&0&$\cdots$\\
    \hline
    2&1&3&&&&&&&\\
    \hline
    3&1&4&&&&&&&\\
    \hline
    4&1&5&&&&&&&\\
    \hline
    $\vdots$&$\vdots$&$\vdots$&&&&&&&\\
    \hline
\end{tabular}
\end{center}
We recall by Lemma~\ref{recursion} that we have the recursive relationship 
\[
d_{n,i}=d_{n-1,i}+d_{n-2,i-1}
\]
for $n\geq 2$ and $i\geq 1$.  We proceed by induction.  Suppose $N,I\geq 2$ and that the dimension formula holds for all $i$ when $n<N$. 
By our recursion and by induction, we have
\begin{eqnarray*}
d_{N,I} & = & d_{N-1,I}+d_{N-2,I-1}\\
 & = & \frac{1}{I!}\prod_{j=1}^{I}(N-1+j-2(I-1))+\frac{1}{(I-1)!}\prod_{j=1}^{I-1}(N-2+j-2(I-2))\\
 & = & \frac{1}{I!}(N-2(I-1))\prod_{j=1}^{I-1}(N+j-2(I-1))+\frac{1}{I!}(I)\prod_{j=1}^{I-1}(N+j-2(I-1))\\
 & = & \frac{1}{I!}\prod_{j=1}^{I}(N+j-2(I-1)),
\end{eqnarray*}
as desired.
\end{proof}

\begin{rmk}\label{lengths and dim to note}
    We note from the proof of the theorem above a few facts for future reference.  
    By our base cases, we have $\ell(\widehat{R(0,t)})=1+1=2$ and $\ell(\widehat{R(1,t)})=1+2=3$.  Taking the Fibonacci sequence $F(n)$ with $F(0)=0$ and $F(1)=1$, we have $F(2)=1$, $F(3)=2$, and $F(4)=3$, so that 
    \begin{eqnarray*}
    \ell(\widehat{R(0,t)})&=&F(3)\\
    \ell(\widehat{R(1,t)})&=&F(4).
    \end{eqnarray*}
    These facts become useful in 
    Proposition~\ref{Fibonacci}.
\end{rmk}

We see in the following corollary that the regularity of $R(n,t)$ is $\left \lfloor n/2\right \rfloor+1$.  For an alternate proof of the regularity of $R(n,t)$ which uses different machinery and more graph-theoretic properties, see the end of this section.

\begin{cor}\label{regcor}
    For $G_n^t\in\mathcal{F}$,
    \begin{equation*}
    \reg R(n,t)=\left\lfloor n/2 \right \rfloor+1.
\end{equation*}
\end{cor}

\begin{proof}
We show that $\reg R(n,t)$ is equal to the top nonzero degree of $\widehat{R(n,t)}$ and that this value agrees with the above.  Since $\widehat{R(n,t)}$ is Artinian by Remark~\ref{art}, it is clear that $\reg \widehat{R(n,t)}$ is the top nonzero degree of $\widehat{R(n,t)}$.  By Theorem~\ref{Hilbert Series mod reg seq}, we know the top nonzero degree is $N$ for some $N\leq n/2+1$, so that $N\leq \left \lfloor n/2\right\rfloor+1$. In fact, the top nonzero degree is $\left\lfloor n/2\right\rfloor+1$, provided $d_{n,\left \lfloor n/2\right \rfloor+1}\neq 0$. 
The only way we have a factor of zero in 
\[
d_{n, \lfloor n/2 \rfloor+1}={\displaystyle \frac{1}{(\lfloor n/2 \rfloor+1)!}\prod_{j=1}^{\lfloor n/2 \rfloor+1}(n+j-2(\lfloor n/2 \rfloor))}
\]
is if $n+j-2(\lfloor n/2\rfloor)=0$, which means $j/2= \lfloor n/2 \rfloor -n/2$.  This can only happen when $j<1$, but $j\geq 1$, so we conclude that $d_{n,\left \lfloor n/2\right \rfloor+1}\neq 0$. Since $\widehat{R(n,t)}$ is an Artinian quotient of $R(n,t)$ by a linear regular sequence, we conclude that
\[
\reg R(n,t)=\reg \widehat{R(n,t)}=\left \lfloor n/2\right \rfloor+1.
\] 
\end{proof}

In the following, we first compute the lengths of the dimension zero rings $\widehat{R(n,t)}$, and then show a closed form for the multiplicity of our original rings $R(n,t)$ by using a Fibonacci relationship between the lengths of the rings $\widehat{R(n,t)}$ and applying Binet's formula for $F(n)$, the $n$th number in the Fibonacci sequence:
\[
    F(n)=\frac{(1+\sqrt{5})^{n}-(1-\sqrt{5})^{n}}{2^{n}\sqrt{5}}.
\]

In the theorem and corollaries which follow, we suppress $t$ for convenience, since the statements are independent of $t$.

\begin{prop}\label{Fibonacci}
The lengths of the rings $\widehat{R(n)}$ satisfy the recursive formula
\[
\ell(\widehat{R(n)})=\ell(\widehat{R(n-1)})+\ell(\widehat{R(n-2)})
\]
for $n\geq 2$.
Consequently, if $F(n)$ is the Fibonacci sequence, with $F(0)=0$
and $F(1)=1$, then 
    \[
    \ell(\widehat{R(n)})=F\left(n+3\right)=\frac{(1+\sqrt{5})^{n+3}-(1-\sqrt{5})^{n+3}}{2^{n+3}\sqrt{5}}.
    \]

\end{prop}

\begin{proof}

Again, we use Notation~\ref{d}. By the recursive relationship from Lemma~\ref{recursion}, since $d_{n,0}=1$ in general, and since $d_{n,i}=0$
in general for $i> n/2 +1$ by Theorem~\ref{Hilbert Series mod reg seq}, we have for $n\geq 2$ that 
\begin{eqnarray*}
\ell(\widehat{R(n)})  =  \sum_{i=0}^{\left\lfloor n/2\right\rfloor +1}d_{n,i}
& = & 1+\sum_{i=1}^{\left\lfloor n/2\right\rfloor +1}\left(d_{n-1,i}+d_{n-2,i-1}\right)\\
& = & \sum_{i=0}^{\left\lfloor n/2\right\rfloor +1}d_{n-1,i}+\sum_{i=0}^{\left\lfloor (n-2)/2\right\rfloor +1}d_{n-2,i}\\
 & = & \ell(\widehat{R(n-1)})+\ell(\widehat{R(n-2)}).
\end{eqnarray*}

Now we show the second statement. For our base cases, we see from Remark~\ref{lengths and dim to note} that  $\ell(\widehat{R(0)})=F(3)=F(0+3)$ and that $\ell(\widehat{R(1)})=F(4)=F(1+3)$.

Now suppose that ${\displaystyle \ell(\widehat{R(n-1)})=F\left(n+2\right)}$
and ${\displaystyle \ell(\widehat{R(n-2)})=F\left(n+1\right)}$. \\ Then
we have 
\begin{equation*}
\ell(\widehat{R(n)})  =  \ell(\widehat{R(n-1)})+\ell(\widehat{R(n-2)})\\
  =  F\left(n+3\right),
\end{equation*}
as desired.  The closed form for $\ell(\widehat{R(n)})$ follows directly from Binet's formula for the Fibonacci sequence.
\end{proof}

\begin{cor}
    For $n\geq 2$, there is an equality of multiplicities
    \[
    e(\widehat{R(n)})=e(\widehat{R(n-1)})+e(\widehat{R(n-2)}).
    \]
\end{cor}

\begin{proof}
    We have established the length of the Artinian rings $\widehat{R(n)}$, and hence the multiplicity $e(\widehat{R(n)})$. 
\end{proof}

\pagebreak

\begin{cor}\label{multcor}
    For $n\geq 2$, there is an equality of multiplicities \[
    e(R(n))=e(R(n-1))+e(R(n-2)).\]  In particular, 
    \[e(R(n))=F\left(n+3\right)=\frac{(1+\sqrt{5})^{n+3}-(1-\sqrt{5})^{n+3}}{2^{n+3}\sqrt{5}}.
    \]
\end{cor}

\begin{proof}
    To obtain the multiplicity of $R(n)$, we look at $\widehat{R(n)}=R(n)/(\overline{X_n})$, which by Remark~\ref{art} and Corollary~\ref{xt} is the Artinian quotient of $R(n)$ by a linear regular sequence.  By a standard result, we may calculate length along the obvious short exact sequences coming from multiplication by elements of our regular sequence to obtain the equality  
    \[
    \Hilb_{R(n)}(t)(1-t)^{d}=\Hilb_{\widehat{R(n)}}(t),
    \]
    where $d$ is the Krull dimension of $R(n)$. Defining multiplicity as in and preceding \cite[Thm 16.7]{Peeva2011}, it follows immediately that
    \[
    e(R(n))=\Hilb_{R(n)}(t)(1-t)^{d}\big\rvert_{t=1}=\Hilb_{\widehat{R(n)}}(1)=\ell(\widehat{R(n)}).
    \]
    We are done by Proposition~\ref{Fibonacci}.
\end{proof}

We reintroduce $t$ and spend the remainder of this section providing an alternate \\graph-theoretic proof for the regularity of $R(n,t)$.

\begin{proof}[Alternate proof of Corollary~\ref{regcor}]

We show that $\reg R(n,t)=\left\lfloor n/2 \right \rfloor+1$ by proving that $\reg I_{G_n^t}=\left\lfloor n/2 \right \rfloor+2$.  We first show that
\[
\reg I_{G_n^t}\leq \left\lfloor n/2 \right \rfloor+2.
\]
We recall by Proposition~\ref{chordal} that the graph $G_n^t$ is chordal bipartite with vertex bipartition $V_1\cup V_2$ of cardinalities 
\begin{eqnarray*}
    |V_1| &=& \left \lfloor \frac{n}{2} \right \rfloor +2\\
    |V_2| &=& \left \lceil \frac{n}{2} \right \rceil +2,
\end{eqnarray*}
and that $G_n^t$ does not have any vertices of degree one by Remark~\ref{graphtaurmk}.  Then by Theorem 4.9 of \cite{BIERMANN2017}, we have  
\begin{equation*}
    \reg I_{G_n^t}\leq \min\left \{\left \lfloor \frac{n}{2} \right \rfloor +2,\left \lceil \frac{n}{2} \right \rceil +2\right \}=\left \lfloor \frac{n}{2} \right \rfloor +2.
\end{equation*}

\noindent We note that we may equivalently prove $\reg R(n,t)\leq \left \lfloor \frac{n}{2} \right \rfloor +1$ by choosing the $\left \lfloor \frac{n}{2} \right \rfloor +2$ edges whose indices are equivalent to zero modulo $4$, one from each row of $L_n^t$, to obtain an edge matching (different from an induced matching, below) and then applying \cite[Th 1]{HH20}.

We now show that $\reg I_{G_n^t}\geq \left\lfloor n/2 \right \rfloor+2$. 
    Since $I_{G_n^t}$ is homogeneous and $\init I_{G_n^t}$ consists of squarefree monomials by Corollary~\ref{initial}, we have by Corollary 2.7 of \cite{conca2018squarefree} that $\reg \init I_{G_n^t} = \reg I_{G_n^t}$, so it suffices to prove that $\reg \init I_{G_n^t}\geq \left\lfloor n/2 \right \rfloor+2$.  The ideal $\init I_{G_n^t}$ can be viewed as the edge ideal of a simple graph, a ``comb” with $n+1$ tines, with consecutive odd variables corresponding to vertices along the spine, as pictured below: 
    
    
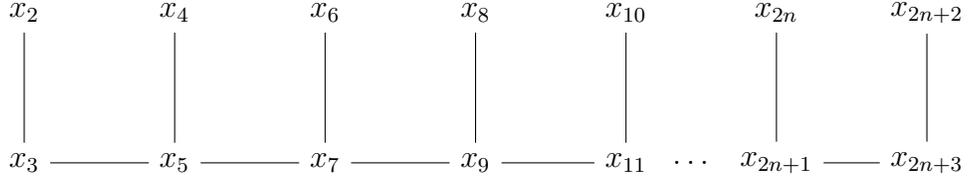
\begin{center}
    \begin{tikzpicture}
	    \definecolor{ddpp}{rgb}{0,.4,.4}
        [scale=.8,every node/.style={circle,fill=none}]
	    \node (n2) at (4,10)  {$x_2$};
        \node (n3) at (4,8)  {$x_3$};
        \node (n5) at (6,8)  {$x_5$};
        \node (n4) at (6,10)  {$x_4$};
        \node (n7) at (8,8)  {$x_7$};
        \node (n6) at (8,10)  {$x_6$};
        \node (n9) at (10,8)  {$x_9$};
        \node (n8) at (10,10)  {$x_8$};
        \node (n11) at (12,8)  {$x_{11}$};
        \node (n10) at (12,10)  {$x_{10}$};
        \node (n13) at (14,8)  {$x_{2n+1}$};
        \node (n12) at (14,10)  {$x_{2n}$};
        \node (n15) at (16,8)  {$x_{2n+3}$};
        \node (n14) at (16,10)  {$x_{2n+2}$};
        
	    \path (n11) -- node[auto=false, fill=none]{\ldots} (n13);

	    \foreach \from/\to in {n2/n3,n4/n5,n6/n7,n8/n9,n10/n11,n12/n13,n14/n15,n3/n5,n5/n7,n7/n9,n9/n11,n13/n15}
        \draw (\from) -- (\to);

    \end{tikzpicture}
    \end{center}
    
    We know from Theorem 6.5 of \cite{Tai-Ha-Van-Tuyl-2008} that the regularity of an edge ideal is bounded below by the number of edges in any induced matching plus one, so we choose $\left \lfloor{n/2}\right \rfloor+1$ edges (tines) corresponding to certain odd variables that create an induced matching.  By beginning with the $x_3$-tine and choosing every other tine corresponding to the variables \begin{equation*}
        x_3,x_{3+4(1)},\ldots,x_{3+4(\left \lfloor{n/2}\right \rfloor)},
    \end{equation*}
    we obtain $\left \lfloor{n/2}\right \rfloor+1$  edges that are an induced matching, so we have 
    \[
    \reg\init I_{G_n^t}\geq \left \lfloor{n/2}\right \rfloor +2,
    \]
    as desired. 
    
    We conclude that $\reg I_{G_n^t}= \left\lfloor n/2 \right \rfloor+2$, and hence that $\reg R(n,t)=\left\lfloor n/2 \right \rfloor+1.$
\end{proof}

\newpage

\bibliographystyle{abbrv}
\bibliography{Paper1.bib}

\end{document}